\documentclass[11pt]{amsart}
\usepackage{amscd,amsmath,amssymb,amsfonts,verbatim}
\usepackage[cmtip, all]{xy}
\usepackage[letterpaper, left=3.7cm, right=3.7cm, top = 2.8cm, bottom = 2.8cm ]{geometry}
\usepackage{comment}
\usepackage{lipsum}
\usepackage{filecontents}
\usepackage[dvipsnames]{xcolor}
\usepackage[colorlinks, pagebackref]{hyperref}
\hypersetup{
	colorlinks=true,
	citecolor=Red,
	linkcolor=MidnightBlue,
	urlcolor=Aquamarine}

\newcommand{\bl}[1]{\color{black}{#1}}

\usepackage{soul}

\newtheorem{thm}{Theorem}[section]
\newtheorem{prop}[thm]{Proposition}
\newtheorem{lem}[thm]{Lemma}
\newtheorem{cor}[thm]{Corollary}

\theoremstyle{definition}

\newtheorem{defn}[thm]{Definition}
\theoremstyle{remark}
\newtheorem{remk}[thm]{Remark}
\newtheorem{remks}[thm]{Remarks}
\newtheorem{const}[thm]{Construction}
\newtheorem{exm}[thm]{Example}
\newtheorem{exms}[thm]{Examples}
\newtheorem{notat}[thm]{Notation}
\newtheorem{setup}[thm]{\bf Setup}
\newtheorem{assmp}[thm]{\bf Assumptions}
\numberwithin{equation}{section}

{\hfill$\square$\end{defn}}
{\hfill$\square$\end{remk}}
{\hfill$\square$\end{remks}}
{\hfill$\square$\end{exm}}
{\hfill$\square$\end{exms}}
{\hfill$\square$\end{notat}}

\newcommand{\thmref}{Theorem~\ref}
\newcommand{\propref}{Proposition~\ref}
\newcommand{\corref}{Corollary~\ref}
\newcommand{\defref}{Definition~\ref}
\newcommand{\lemref}{Lemma~\ref}

\newcommand{\remref}{Remark~\ref}

\newcommand{\secref}{Section~\ref}

\newcommand{\setref}{Setup~\ref}

\newcommand{\sA}{{\mathcal A}}

\newcommand{\sF}{{\mathcal F}}

\newcommand{\sM}{{\mathcal M}}

\newcommand{\sO}{{\mathcal O}}

\newcommand{\sU}{{\mathcal U}}
\newcommand{\sV}{{\mathcal V}}
\newcommand{\sW}{{\mathcal W}}
\newcommand{\sX}{{\mathcal X}}
\newcommand{\sY}{{\mathcal Y}}
\newcommand{\sZ}{{\mathcal Z}}

\newcommand{\C}{{\mathbb C}}
\newcommand{\D}{{\mathrm D}}

\newcommand{\LL}{{\mathrm L}}

\renewcommand{\P}{{\mathbb P}}
\newcommand{\Q}{{\mathbb Q}}

\newcommand{\Z}{{\mathbb Z}}

\newcommand{\fm}{{\mathfrak m}}
\newcommand{\fg}{{\mathfrak g}}

\newcommand{\vir}{{\rm vir}}

\newcommand{\CH}{{\rm CH}}

\newcommand{\surj}{\twoheadrightarrow}
\newcommand{\inj}{\hookrightarrow}

\newcommand{\codim}{{\rm codim}}
\newcommand{\rank}{{\rm rank}}

\newcommand{\Spec}{{\rm Spec \,}}

\newcommand{\Coh}{{\rm Coh}}
\newcommand{\QCoh}{{\rm QCoh}}

\newcommand{\GG}{{\rm G}}
\newcommand{\KK}{{\rm K}}
\newcommand{\CA}{{\rm A}}

\newcommand{\Sch}{{\operatorname{\mathbf{Sch}}}}

\newcommand{\Sm}{{\mathbf{Sm}}}

\newcommand{\DM}{{\operatorname{\mathcal{DM}}}}
\newcommand{\BG}{{\operatorname{\rm{BG}}}}
\newcommand{\BGH}{{\operatorname{\rm{B( G \times H)}}}}
\newcommand{\BH}{{\operatorname{\rm{BH}}}}

\newcommand{\ds}{{/\kern-3pt/}}

\newcommand{\res}{{\operatorname{res}}}

\renewcommand{\dim}{\text{\rm dim}}

\newcommand{\tuborg}{\left\{\begin{array}{ll}}
\newcommand{\sluttuborg}{\end{array}\right.}

\newcounter{elno}

\newcounter{elno-abc}   

\newcounter{elno-abc-prime}   

\usepackage{epigraph}

\begin{document}
\sloppy
    
\title[Equivariant Virtual Classes]
{Virtual Equivariant Grothendieck-Riemann-Roch Formula}
\author{Charanya Ravi}
\address{Universit\"{a}t Regensburg, Universit\"{a}tsstr. 31, 93040 Regensburg, Germany}
\email{charanya.ravi@mathematik.uni-regensburg.de}

\author{Bhamidi Sreedhar}
\address{Korea Institute For Advanced Study, 85 Hoegi-ro,  Seoul, Republic Of Korea}
\email{sreedhar@kias.re.kr}





\begin{abstract}
	For a $G$-scheme $X$ with a given equivariant perfect obstruction theory, 
	we prove a virtual equivariant Grothendieck-Riemann-Roch formula, 
	this is an extension of a result of Fantechi-G\"{o}ttsche (\cite{FG}) to the equivariant context.
	We also prove a virtual non-abelian localization theorem for schemes over $\C$ with proper actions.

\end{abstract} 
\maketitle
\setcounter{tocdepth}{2}{}
\tableofcontents  
{\hypersetup{linkcolor=black} \tableofcontents}

\section{Introduction}\label{sec:Intro}

As several moduli spaces that one encounters in algebraic geometry have a well-defined expected dimension, 
one is interested in constructing a  fundamental class of the expected dimension in its Chow group. 
Interesting numerical invariants like Gromov-Witten invariants and  Donaldson-Thomas invariants
are obtained by integrating certain cohomology classes over such a fundamental class.
This motivated the construction of virtual fundamental classes. 
Given a moduli space with a perfect obstruction theory, such classes were defined for complex analytic spaces by Li and Tian in \cite{LT} 
and in the algebraic sense for Deligne-Mumford stacks by Behrend and Fantechi in \cite{BF}.  

Given a Deligne-Mumford stack $\sX$, Behrend and Fantechi construct an algebraic stack $\mathfrak{C}_{\sX}$ over $\sX$,
called the intrinsic normal cone of $\sX$. The choice of a perfect obstruction theory $E^\bullet$ over $\sX$
is then used to defined a vector bundle stack $\mathfrak{E}$ over $\sX$, which contains $\mathfrak{C}_{\sX}$ as a closed substack.
The virtual fundamental class associated to $E^\bullet$ is a class in the Chow group of $\sX$
obtained by intersecting $\mathfrak{C}_{\sX}$ with the zero section of $\mathfrak{E}$. 
Following a similar construction, Lee in \cite{YPLee} defined the virtual structure sheaf 
in the Grothendieck group of coherent sheaves of $\sX$. 
For schemes, the Riemann-Roch transformation provides a natural isomorphism between $K$-theory of coherent sheaves 
and Chow groups with rational coefficients
and it is natural to ask if there is a relation between the virtual fundamental class and the image of the 
virtual structure sheaf under the Riemann-Roch transformation. 
This question goes back to the definition of the virtual fundamental class for quasi-manifolds by Kontsevich (see \cite[\S~1.42]{KO}), 
where the class is defined via the Riemann-Roch map using the definition of the virtual structure sheaf.
In \cite{FG}, Fantechi and G\"{o}ttsche prove that for a scheme with a perfect obstruction theory, 
the image of the virtual structure sheaf is indeed the virtual fundamental class twisted by the Todd class of the virtual tangent bundle, 
thus verifying the formula proposed by Kontsevich.
A similar result in the context of $dg$-manifolds was obtained previously by Ciocan-Fontanine and Kapranov in \cite{CK}. 

A theory of Chow groups for algebraic stacks has been developed by various authors. 
For Deligne-Mumford stacks this was introduced by Vistoli in \cite{Vistoli}, 
for quotient stacks by Edidin-Graham in \cite{EG} and for algebraic stacks which admit stratifications by quotient stacks, 
this was developed by Kresch in \cite{Kresch}. 
Quotient stacks which are obtained as the quotient of a scheme (or an algebraic space) 
with the action of linear algebraic group
turn out to be some of the most well understood algebraic stacks 
as one can approach questions related to them as equivalent problems in equivariant geometry. 
In this text, we study the virtual classes associated to equivariant perfect (relative) obstruction theories, 
that is  perfect obstruction theory associated to a morphism of stacks $[X/G] \to [Y/G]$, 
where the the map is induced by a $G$-equivariant morphism between $G$-schemes $X \to Y$.
In \cite{EG1}, Edidin-Graham define an equivariant Riemann-Roch transformation $\tau^G_{-}$ 
from equivariant $K$-theory $\GG_0(G,-)_{\Q}$ 
to the equivariant Chow groups $\CH_*^G (-)_{\Q}$. 
We show that for a  scheme with group action and a given equivariant perfect obstruction theory, 
the virtual structure sheaf and virtual fundamental class satisfy Kontsevich's formula. 
More precisely, we prove the  following:
\begin{thm}({see \thmref{thm:main}})\label{intro: thm:main}
	Let $X, Y \in \Sch^G_k$ and $\tilde{f}: X \to Y$ be a $G$-equivariant morphism such that $Y$ is smooth, 
	$G$-equivariantly connected and pure dimensional. 
	Suppose that there exists a $G$-equivariant closed immersion $X \inj M$ for some $M \in \Sm_k^G$ 
	such that the assumptions in \ref{assmp} are satisfied.
	Then for an equivariant perfect relative obstruction theory $E^\bullet \to \LL_{X/Y}$ on $X$ with respect to $Y$ which admits a global resolution:
	\begin{equation}
	\tau^G_{X} (\sO^\vir_{[G, X, E^\bullet]}) = Td^G(T^\vir_{[X/Y, E^\bullet]}) \cap [G, X^\vir, E^\bullet]
	\end{equation}
	in $\CH_*^G(X)_\Q$.
\end{thm}
Here $\sO^\vir_{[G, X, E^\bullet]}$, $[G, X^\vir, E^\bullet]$ and $T^\vir_{[X/Y, E^\bullet]}$
denote the virtual structure sheaf, the virtual fundamental class and the virtual tangent bundle associated to the
$G$-equivariant perfect relative obstruction theory $E^\bullet$.
This is  an extension of \cite[Proposition 3.4]{FG} to the case of quotient stacks. 
As an immediate consequence of the above result, we prove the following virtual equivariant Grothendieck Riemann-Roch theorem 
which generalizes \cite[Theorem 3.5]{FG}. 

\begin{cor}(see \corref{cor:EVGRR}) \label{cor-GRR,HRR}
	Let $X$ be a $G$-scheme satisfying the assumptions in \ref{assmp} and
	let $E^\bullet \to \LL_{X}$ be an equivariant perfect obstruction theory on $X$ which admits global resolution
	and let $\alpha \in \KK_0(G, X)$.
	\begin{enumerate}
		\item{(Virtual Equivariant Grothendieck-Riemann-Roch).}
		Let  $f: X  \rightarrow Y  \in  \Sch_k^G$ where $f$ is proper and $Y$ is smooth and $G$-quasi-projective.
		Then:
		\begin{equation} 
		ch^G(f_*(\alpha \otimes \sO^\vir_{[G, X, E^\bullet]}) ) ~.~ Td^G(T_Y) \cap [Y] = \\ 
		f_*(ch^G(\alpha) ~.~ Td^G(T_{[X,E^\bullet]}^\vir) \cap [G, X^\vir,E^\bullet])
		\end{equation}
		in $\CH^G_*(X)_\Q$.\\
		\item{(Virtual Equivariant Hirzebruch-Riemann-Roch).} If $X$ is a complete $G$-scheme, then:
		\begin{equation} 
		\chi^G_\vir(X, E^\bullet, \alpha) = \int_{[G,X^\vir, E^\bullet]} ch^G(\alpha) ~.~ Td^G(T^\vir_{[X, E^\bullet]}).
		\end{equation}
	\end{enumerate}
\end{cor}

There are several interesting examples of perfect obstruction theories on schemes which are equivariant with respect to an underlying group action.
For a smooth complex projective toric threefold $Y$ with its natural split torus action, it has been recently shown in the work 
of Ricolfi that the natural perfect obstruction theories on the Hilbert scheme $\mathrm{Hilb}^n(Y)$ of $n$-points on $Y$, and the 
Quot scheme of length $n$ quotients of a torus equivariant exceptional locally free sheaf on $Y$  
are all in fact equivariant with respect to the torus action (see \cite[Theorem B, Example 4.6]{Ri20}). 
On a smooth projective complex curve $X$, the moduli stack of $G$-local systems on $X$ for a complex reductive group $G$
is a quotient stack with an equivariant perfect obstruction theory (see e.g. \cite{BN16}).

The next goal of this paper is to study the relation between the virtual classes of a quotient stack and its inertia stack.
The equivariant Riemann-Roch map fails to be an isomorphism even in simple cases of trivial actions.
For example, for $X= \Spec \C$ with a trivial $G = \Z/n\Z$ action, while $\GG_0(G,X)_{\Q}$ is an $n$-dimensional $\Q$-vector space 
$\CH_*^G (X)_{\Q}$ is a $1$-dimensional $\Q$-vector space.
However this anomaly was rectified by Edidin-Graham in \cite{EG2} for separated quotient Deligne-Mumford stacks over $\C$ by 
taking complex coefficients and by 
constructing a Riemann-Roch transformation which 
maps to the Chow group of the inertia.
For a $G$-scheme $X$ with finite stabilizers, let $$I\tau_X^G: \GG_0(G,X)_{\C} \to \CH_*^G(I_X)_{\C},$$ denote this map,
where $I_X$ denotes the inertia scheme of $X$. 
The map $I\tau^G_X$ factors through the geometric part of the $K$-theory of the inertia scheme, 
i.e., $$I\tau^G_X: \GG_0(G,X)_{\C} \xrightarrow{\vartheta^G_X} (\GG_0(G,I_X)_{\C})_{\fm_1} \xrightarrow{\tau^G_{I_X}} \CH_*^G(I_X)_\C,$$
where  the map  $\vartheta^G_X$ is the Atiyah-Segal map and $\fm_1$ is the augmentation ideal.  
For finite group actions the Atiyah-Segal map was introduced by Vistoli in \cite{V} 
and studied by Edidin-Graham in \cite{EG2} and \cite{EG4} in the construction of their
 Riemann-Roch transformation to the Chow group of the inertia. 
  In \cite{KS20}, whose terminology
 is followed in this paper to describe this map,
 the second author and Krishna  extend 
 the Atiyah-Segal map to higher $K$-theory  for quotient Deligne-Mumford stacks over $\C$ and prove that it is an isomorphism.
 In Section \ref{Sec:Atiyah-Segal}, we describe the image of the virtual structure sheaf under the Atiyah-Segal map.
 And as a corollary, we give an explicit formula for the virtual structure sheaf under the Riemann-Roch transformation to inertia
 in terms of Chern characters and Todd classes. This can be seen as a virtually smooth analog of \cite[Theorem 6.5(c)]{EG4}. 
 Note that in Theorem \ref{thm-AS} and Corollary \ref{cor-AS}, all $K$-groups and Chow groups are taken with complex coefficients.

\begin{thm}({ see \thmref{thm:Virtual-sh-AS}}) \label{thm-AS}
	Let $X \in \Sch^G_\C$ such that $G$ acts properly on $X$. Suppose that there exists a $G$-equivariant closed immersion $X \inj M$ for some $M \in \Sm_\C^G$. Let $E^\bullet \to \LL_{X/k}$ be a $G$-equivariant perfect obstruction theory which admits a global resolution and let 
	$\tilde{E}^\bullet \to \LL_{I_X/k}$ be the induced $G$-equivariant perfect obstruction theory  on $I_X$ (see \lemref{lem:GPInertia}). 
		Then under the Atiyah-Segal isomorphism
		$\vartheta^G_X: \GG_0(G,X) \to \GG_0(G,I_X)_{\fm_1}$, we have the following:
		$$ \vartheta^G_X(\sO^\vir_{[G, X, E^\bullet]}) = 
		\underset{\psi \in \Sigma^G_X} \sum \frac{\sO^\vir_{[G,I^\psi_X, \tilde{E}^\bullet]}}{\Lambda_{-1}([N_{[i^\psi, E^\bullet]}^{\vir^\vee}])}.
		$$
	\end{thm}
Here $N_{[i^\psi, E^\bullet]}^{\vir^\vee}$ denotes the virtual conormal bundle of the inertia (see Section \ref{sec:AS-Virtual sheaf}).
The above theorem is closely related to the virtual localization theorems for  torus actions for Chow groups due to Graber-Pandharipande in \cite{GP99}
and for $K$-theory due to Qu in \cite{QU}.
The perfect obstruction theory induced on the inertia scheme is obtained using the proof of \cite[Lemma 1]{GP99}. 
Combining the above theorem with  \thmref{intro: thm:main} we have the following corollary. 

\begin{cor} (see \corref{cor:Inertia-RR}) \label{cor-AS}
Under the notations of Theorem \ref{thm-AS} further  suppose  assumptions in \ref{assmp} are satisfied. Then we have the following equality in $\CH_*^G(I_X)$:
	$$
	I\tau^G_X (\sO^\vir_{[G, X, E^\bullet]}) = \underset{\psi \in \Sigma^G_X} \sum ch^G(\Lambda_{-1}([N_{[i^\psi, E^\bullet]}^{\vir^\vee}])^{-1}
	Td^G(T^\vir_{[I^\psi_X, \tilde{E}^\bullet]}) \cap  [G,I_X^{\psi^\vir}, \tilde{E}^\bullet].
	$$
\end{cor}

We briefly describe some easy consequences of these results. Let $X$ be a scheme with a perfect obstruction theory $\phi: E^\bullet \to \LL_{X/k}$. 
In \cite{Si}, Siebert shows that the virtual fundamental class in the Chow group is independent of the map $\phi$ and only depends on $X$ 
and the class of $E^\bullet$ in $\KK_0(X)$. As noted in \cite{ NMR20}, for a scheme with $G$-action and an equivariant perfect obstruction theory, 
Siebert's proof goes through and one has a similar statement for the equivariant virtual fundamental class. 
On the other hand in  \cite{Th18}, Thomas defines a $K$-theoretic analogue of the Fulton class and
shows that the virtual structure sheaf is independent of $\phi$ 
and only depends on the class of $E^\bullet$ by showing that the virtual structure sheaf can be calculated in terms of the $K$-theoretic Fulton class.
Note that in  \corref{cor-AS} the terms $ch^G(\Lambda_{-1}([N_{[i^\psi, E^\bullet]}^{\vir^\vee}])$ and $	Td^G(T^\vir_{[I^\psi_X, \tilde{E}^\bullet]})$ 
are independent of the map $\phi$ and only depend on $[E^\bullet]$ in $\KK_0(G,X)$ (see \secref{Sec:Atiyah-Segal} for the definitions). 
Now  as  $I\tau_X$ is an isomorphism, it follows from  the equivariant form of Siebert's result for equivariant Chow groups 
and \corref{cor-AS} that the virtual structure sheaf $[\sO^\vir_{[G, X, E^\bullet]}]$ in $\GG_0(G, X )_{\C}$ 
is also independent of the map $\phi$.

 For a complete $G$-scheme $X$ with proper action satisfying the assumptions of \corref{cor-AS},
we can combine the results of Theorem \ref{thm-AS} and Corollary \ref{cor-GRR,HRR} (2) applied to the Riemann-Roch transformation 
for the inertia stack to obtain the following form of the virtual Kawasaki-Riemann-Roch formula:
$$
\chi^G_\vir(X, E^\bullet, \alpha) = \int_{[G,I_X^\vir, \tilde{E}^\bullet]} ch^G\left(\frac{\alpha_{I_X}}{\Lambda_{-1}([N_{[i^\psi, E^\bullet]}^{\vir^\vee}]}\right) ~.~ Td^G(T^\vir_{[I_X, \tilde{E}^\bullet]}),
$$
for a choice of an equivariant perfect obstruction theory  $E^\bullet \to \LL_{X}$ on $X$ which admits global resolution
and for any $\alpha \in \KK_0(G,X)$. Here $\tilde{E}^\bullet \to \LL_{I_X/k}$ is the induced equivariant perfect obstruction theory on $I_X$
and $\alpha_{I_X}$ denotes the twisted image of $\alpha$ in $\KK_0(G,I_X)$
and is defined as $\mu (t^{-1}(i^!(\alpha)))$, where $i^!$, $t$ and $\mu$ denote the pullback on $K$-theory (see \cite[\S~ 5.2]{EG4}),
the twist and the Morita maps, respectively.
A virtual Kawasaki-Riemann-Roch formula in terms of the Euler characteristic was obtained earlier by Tonita in \cite{Tonita},
which has applications in studying the quantum $K$-theory of complex manifolds.

We briefly describe other works in literature that are closely related to the results obtained in this paper. 
The relationship between the virtual structure sheaf and the virtual fundamental class 
has been studied by several authors. 
For torus actions a version of this relationship was established by Thomas in \cite{Thomas}, in the setting of Bredon style homology  theories
this has been established by Joshua in \cite{joshua} and in the generality of quasi-smooth derived algebraic stacks this has been established recently 
by Khan \cite[\S~3.5]{khan} for \'{e}tale motivic Borel–Moore homology theories. 
A virtual Atiyah-Segal type isomorphism has been established using the theory of 
$\theta$-stratifications for quasi-smooth derived stacks by Halpern-Leistner (see \cite{HL}). 
 A $K$-theoretic virtual localization theorem for finite group actions has also been parallelly proved in \cite{Gu20}.

\subsubsection*{\bf Outline of the paper}
In \secref{sec:preliminaries} we first set-up the notations and conventions that would be used in the rest of this text. 
To keep the paper reasonably self contained we recall the definitions of perfect relative obstruction theories and 
virtual classes associated to them. We also recall the construction of the equivariant Riemann-Roch map from \cite{EG}, 
as the precise construction plays an important role in the proof of our main theorems. \secref{sec:VCRR} is devoted to the proof of 
\thmref{thm:main}. In \secref{Sec:Atiyah-Segal} we study the image of the virtual structure sheaf under the Atiyah-Segal map. 
We recall the definition of the Atiyah-Segal map from \cite{EG4} and \cite{KS20} in \secref{sec:Atiyah-Segal-map}.
 In \secref{sec:Morita} we discuss the construction of  the induced perfect obstruction theory on the inertia scheme 
and prove that the Morita isomorphism respects virtual classes. In \secref{sec:AS-Virtual sheaf} 
we study the image of the virtual structure sheaf under various maps in the definition of the Atiyah-Segal map and 
finally prove the main theorem of this section which is \thmref{thm:Virtual-sh-AS}.

\subsubsection*{\bf Acknowledgements}
We thank Richard Thomas for his inputs and suggestions.
The first author would like to thank Adeel Khan for several useful conversations on the subject.
The second author thanks Amalendu Krishna for suggesting him to think about an equivariant analogue of \cite{FG} and Bumsig Kim for discussions. 
C.R. was supported by SFB 1085 Higher Invariants, Universit\"{a}t Regensburg.
B.S. was supported by KIAS Individual Grant MG062803 at Korea Institute for Advanced Study.

\section{Preliminaries}\label{sec:preliminaries}

In this section we recall some preliminaries related to virtual classes and equivariant Chow groups. We first begin by fixing the notations that would be 
used in the rest of this text. 

\subsection{Notations and conventions}\label{subsec:notations}
In this note we assume that our base is a field $k$. Let $\Sch_k$ denote the category of separated schemes of finite type over $k$. $\Sm_k$  
shall denote the full  subcategory of smooth schemes in $\Sch_k$. For $X \in \Sch_k$ we let $\CH_j(X)$ 
denote the $j$th Chow group of $X$. 

A stack shall mean an algebraic stack of finite type over $k$. A Deligne-Mumford stack over $k$ will be abbreviated to a $\DM$-stack. Let $f: \sX \to \sY$ 
be a morphism of stacks, $f$ is called a {\sl $\DM$-type} morphism if for any morphism $S \to \sY$ from an algebraic space $S$, the fiber product 
$S \times_{\sY} \sX$ is a $\DM$-stack. 

A group scheme $G$ will refer to a linear algebraic group over $k$, that is a smooth affine group scheme over $k$. For a group scheme $G$, let $\fg$ 
denote the lie algebra of $G$, which is viewed as a $G$-representation via the adjoint representation of $G$.  
Let $\Sch^G_k$   denote the category of schemes over $k$ 
with  $G$-action. 
For any $X \in \Sch_k^G$ there is an associated quotient algebraic stack $\sX=[X/G]$,
we let $\BG$ denote the classifying stack $[\Spec k/G]$ of $G$. 
For a ($G$-equivariant) coherent sheaf $V$ over $\Spec k$, a finite dimensional $k$-vector space (representation of $G$ over $k$),
by abuse of notation we also use $V$ to denote the corresponding ($G$-equivariant) vector bundle over $\Spec k$.


For a stack $\sX$, let $\QCoh(\sX)$ and $\Coh(\sX)$ denote the categories of quasicoherent $\sO_{\sX}$-modules and 
coherent $\sO_{\sX}$-modules on $\sX$ respectively and
let $\GG(\sX)$ denote the Quillen $K$-theory spectrum of the exact category $\Coh(\sX)$. 
For $X \in \Sch^G_X$,  we let $\Coh^G(X)$ denote the abelian category of $G$-equivariant coherent $\sO_{X}$-modules on $X$
and the $K$-theory spectrum of $\Coh^G(X)$ is denoted by $\GG(G,X)$. 
$\GG_i(\sX)$ and $\GG_i(G,X)$ are defined to be the $i$th homotopy groups of $\GG(\sX)$ and $\GG(G,X)$ respectively, for $i = 0$,
$\GG_0(\sX)$ and $\GG_0(G,X)$ are therefore the Grothendieck groups of the exact categories $\Coh(\sX)$ and $\Coh^G(X)$ respectively.
Any coherent $\sO_{\sX}$-module $\sF \in \Coh(\sX)$ (respectively $\Coh^G(X)$) defines an element of $\GG_0(\sX)$ (respectively $\GG_0(G,X)$) 
which is denoted by $[\sF]$.
For a $G$-scheme $X$, descent along the $G$-torsor $X \to \sX:= [X/G]$ 
induces an equivalence of categories $\Coh^G(X) \xrightarrow{\simeq} \Coh(\sX)$, which induces a homotopy equivalence of spectra 
$\GG(G,X) \xrightarrow{\simeq} \GG(\sX)$. 
$\KK(\sX)$ and $\KK(G,X)$ are defined to be the $K$-theory spectrum of the exact categories of 
locally free $\sO_{\sX}$-modules and $G$-equivariant locally free $\sO_{X}$-modules
respectively which are again homotopy equivalent. And similar to the case of coherent modules the $0$th homotopy groups are the 
Grothendieck groups of locally free sheaves.
There is a natural map of spectra $\KK(\sX) \to \GG(\sX)$ which is not in general an equivalence.

For a $G$-scheme $X$, let $X/G$ denote the geometric quotient scheme of $X$ by $G$ \cite[Definition 0.6]{GIT}, provided it exists.  Let $m: G \times X \to X $
denote the map defining the action of $G$ on $X$. 
Recall that $G$ is said to act freely on $X$ if the map $(m, pr_X): G \times X \to X \times X$  is a closed immersion, where $pr_X$ denotes the projection to $X$.
Suppose that $G$ acts freely on $X$ and the geometric quotient $X/G$ exists as a scheme, then $X \to X/G$ is a $G$-torsor 
(see \cite[Proposition 0.9]{GIT}) and therefore we have
a canonical equivalence of categories between $\Coh(X/G)$ and $\Coh^G(X)$ 
which further induces an isomorphism between the Grothendieck groups and
we denote by $\mathrm{Inv}^G_X$ the natural isomorphism $\GG_0(G,X) \rightarrow \GG_0(X/G)$.

We note that we follow the cohomological notation for complexes, that is for a complex $(E^{\bullet}, d^{\bullet})$, 
we have the differential $d^i: E^i \to E^{i+1}$. 
Let $h^i(E^{\bullet})$ denote the $i^{\text{th}}$-cohomology sheaf of the complex $E^{\bullet}$.  
For a morphism of complexes $\phi: E^{\bullet} \rightarrow F^{\bullet}$, $h^i(\phi)$ 
will denote the induced map on cohomologies $h^i(\phi): h^i(E^{\bullet}) \rightarrow h^i(F^{\bullet})$. 
For any additive category $\sA$, $D(\sA)$ denotes the derived category of $\sA$ and 
$D^b(\sA)$ denotes the derived category of bounded complexes on $\sA$.
For $E^\bullet \in D^b(\Coh(\sX))$ one can naturally define $[E^\bullet]= \Sigma_{i} (-1)^i[ E^i]$  as an element of  $\GG_0(\sX)$ and an element of $\KK_0(\sX)$ if $E^\bullet$ is a complex of locally free sheaves. 

Let $f: \sX \rightarrow \sY$ be a morphism of algebraic stacks then $\LL_{\sX/\sY}$ will denote the relative cotangent complex in $D^b(\Coh(\sX))$, 
for notational convenience this would also be denoted by $\LL_f$. When $\tilde{f}: X \to Y$ is a $G$-equivariant map of $G$-schemes, 
the relative cotangent complex is infact a complex of $G$-equivariant sheaves and $\LL_{X/Y}$ 
or $\LL_{\tilde{f}}$ in this case would refer to this complex as an object of $D^b(\Coh^G(X))$. 
Let $f: \sX := [X/G] \to \sY := [Y/G]$ be the map induced by $\tilde{f}$, then under the equivalence
$D^b(\Coh^G(X)) \simeq D^b(\Coh(\sX))$, the complex $\LL_{X/Y}$ corresponds to $\LL_{\sX/\sY}$.


\subsection{Equivariant Chow groups} 

Equivariant Chow groups were defined by Edidin-Graham in \cite{EG} based on the construction of the  Chow groups of classifying spaces by Totaro 
\cite{Totaro}.  
In this subsection we briefly recall the  definition of equivariant Chow groups and the construction of the equivariant Riemann-Roch transformation, 
for further details see \cite{EG},\cite{EG1}. Some of the definitions are standard and we recall them here primarily to fix notations which would be 
repeatedly used in what follows. 
\subsubsection{Equivariant Chow Groups} \label{Sec:Eq-Chow}
Let $G$ be a linear algebraic group of dimension $g$ and let $X \in \Sch_k^G$ be a scheme with $\dim (X) = n$.
Let $V$ be an $l$-dimensional representation of $G$ over $k$ and let $U$  be a $G$-invariant open subset of $V$ such that $G$ acts freely on $U$. 
The diagonal action on $X \times U$ is also free, so there is a quotient in the category of algebraic spaces $X \times U \to X \times^G U$, 
which in general may not be a scheme. 
The pair $(V,U)$ is called an $l$-dimensional {\sl good pair} for an integer $j$ if $\codim(U,V) > j$.  
\begin{defn}\label{defn:Chow Groups}\cite[Definition-Proposition 1]{EG}
The $i$th {\sl equivariant Chow group} of $X$ is defined as  $\CH_i^G(X) : =  \CH_{i+l-g} (X \times^G U)$, where $(V,U)$ is an $l$-dimensional 
good pair for the integer $n-i$ and $\CH_p$ is the usual Chow group. We write $\CH_*^G(X) = {\underset{-\infty< i \leq n} \prod} \CH_i^G(X)$.
\end{defn}

The definition of equivariant Chow groups in \defref{defn:Chow Groups} is independent of the choice of a good pair $(V,U)$ 
up to unique isomorphisms as shown in \cite[Definition-Proposition 1]{EG}. 
Equivariant Chow groups have the same functoriality as ordinary Chow groups for equivariant proper, flat, smooth, 
regular embedding and l.c.i morphisms.

\begin{remk}
It is shown in \cite[Proposition 16]{EG} that for a $G$-scheme $X$ the groups $\CH_{i}^G(X)$
depend only on the stack $[X/G]$ and not on 
its presentation as a quotient. More precisely, 
suppose $\sZ$ is an algebraic stack which has two presentations $[X/G]$ and $[Y/H]$ as a quotient stack, 
where $X \in \Sch_k^G$ and $Y\in \Sch_k^H$, for group schemes $G$ and $H$. Then there is a natural isomorphism
$\CH_{i+g}^G(X) \simeq \CH_{i+h}^H(Y)$, where $g =\dim(G)$ and $h = \dim(H)$.
\end{remk}

\begin{remk}
Chow groups for more general algebraic stacks have been defined by Kresch in \cite{Kresch}. 
For $\sX$ an algebraic stack, let $\CA_*(\sX)$ denote the Chow groups defined in {\it loc.cit}. In particular, if $\sX= [X/G]$ is in fact a quotient stack it follows from \cite[Theorem 2.1.12]{Kresch}, that there is a natural map $\CH_*^G(X) \to \CA_*(\sX)$ which is a graded isomorphism.  
The groups $\CH_*^G(X)$ are denoted by $\widehat{\CA}_*([X/G])$ and are defined as a limit over the naive Chow groups of vector bundles over the 
stack (see \cite[Definition 2.1.4]{Kresch}). The reason the two definitions coincide is due the fact that the limit stabilizes and further 
vector bundles which are pullbacks from $\BG$ form a cofinal system (see \cite[Remark 2.1.7, Remark 2.1.17]{Kresch}).
\end{remk}

\subsection{Equivariant Riemann-Roch theorem}

We let
$\tau_{(-)}: \GG_0(-)_\Q \to \CH_*(-)_\Q$
denote the Riemann-Roch transformation constructed by Baum-Fulton-Macpherson in \cite{BFM}.
An equivariant analogue of the Riemann-Roch transformation, denoted $\tau^G$, was constructed in \cite{EG1} and for $X \in \Sch_k^G$
it defines a map
\begin{equation}
\tau^G_X : \GG_0(G,X)_\Q \rightarrow \CH_*^G(X)_\Q,
\end{equation}
from equivariant $G$-theory to equivariant Chow groups. In this section we recall briefly the construction of this
map in terms of the Riemann-Roch map of \cite{BFM} as this would be key to our proofs of the equivariant
virtual Riemann-Roch theorems.
  
If $G$ acts freely on a scheme $X$, then there are natural isomorphisms 
$G_0(G,X) \xrightarrow{\simeq} G_0(X/G)$ and $\CH_*^G(X) \xrightarrow{\simeq} \CH_*(X/G)$ and  these isomorphisms identify 
$\tau^G_X$ and $\tau_{X/G}$. Although unlike $\tau_{(-)}$,  $\tau^G_X$ need not be an isomorphism in general.

\subsubsection{Equivariant Riemann-Roch map}\label{para:equivariantRR}
Let $X \in \Sch_k^G$ be of dimension $n$, let $i$ be an integer and let $(V,U)$ be an $l$-dimensional good pair for $n-i$.
Let $j: U \inj V$ denote the open immersion and let $j_X: X \times U \inj X \times V$ denote the open immersion given by base change. 
Let $\bar{p}_X: X \times V \to X$ and $p_U: X \times U \to U$ denote the natural projection maps. 
As $X \times U \times V$ is a $G$-equivariant vector bundle over $X \times U$, 
it descends to define a vector bundle $X \times^G (U \times V)$ over $X \times^{G} U$. 
The Todd class $Td(X \times^G (U \times V))$  of the vector bundle 
$X \times^G (U \times V)$ over $X \times^{G} U$, defines an invertible map from $\CH_*(X \times^G U)_\Q \to \CH_*(X \times^G U)_\Q$.
The equivariant Riemann-Roch map $\tau^G_X: \GG_0(G,X)_\Q \to \CH_*^G(X)_\Q$ 
is determined by  defining its projection to $\CH_i^G(X)_\Q$ for each $i$. 
For a given $i$, the $i$th component of  $\tau^G_X$ is defined by the following commutative diagram: 
\begin{equation}\label{eqn:equivariantRR}
\xymatrix{
                \GG_0(G, X \times V)_\Q \ar[r]^{j_X^*} & \GG_0(G, X \times U)_\Q  \ar[r]_{\mathrm{Inv}^G_{X \times U}}^{\simeq} 
                & \GG_0(X \times^G U)_\Q  \ar[dd]^{\left(\frac{\tau_{X \times^G U}}{Td(X \times^G (U \times V)) }\right)_{i+l-g}} \\\\
                \GG_0(G,X)_\Q \ar[rruu]_{s_U}\ar[uu]^{\bar{p}_X^*}\ar[rr]^{(\tau^G_X)_i} && \CH_{i+l-g}(X \times^G U)_\Q =: \CH^G_{i}(X)_\Q.
}
\end{equation}

The left vertical map $\bar{p}_X^*$ is an isomorphism by homotopy invariance of $G$-theory, $j_X^*$ is the flat pullback, and
the right vertical map is the $(i+l-g)$th projection of the map given by 
the Riemann-Roch map $\tau_{X \times^{G}U}$ multiplied by the inverse of $Td(X \times^G (U \times V)) $. 
Let $s_U: G_0(G,X)_\Q \to G_0(X \times^G U)_\Q$ denote the composition of maps $\mathrm{Inv}^G_{X\times U} \circ p_X^*$,
where $p_X := \bar{p}_X \circ j_X : X \times U \to X$ is the projection map.
Therefore $(\tau^G_X)_i$ is defined as $\left(\frac{\tau_{X \times^G U}}{Td(X \times^G (U \times V)) }\right)_{i+l-g} \circ s_U$, 
and we can define $\tau^G_X$ by making an appropriate choice of good pair $(V,U)$ for each $i$.
And it follows from \cite[Proposition 3.1]{EG1} that this definition of $\tau^G$ is independent of the choices of good pairs.

\subsection{Perfect obstruction theories and virtual classes}
We recall the notion of relative perfect obstruction theory for $\DM$-type morphisms and the virtual classes with respect to them.
The main references are \cite{BF} and \cite{Manolache}.

\begin{defn}\label{defn:obstructiontheory} \cite[Definition 4.4]{BF}\cite[Definition 2.5]{Manolache}
Let $f: \sX \to \sY$ be a $\DM$-type morphism of algebraic stacks. An {\sl obstruction theory} on the stack $\sX$ relative to $\sY$  would consist of the following data:
\begin{enumerate}
	\item A complex $E^{\bullet} \in \D(\QCoh(\sX))$  such that $h^i(E^{\bullet}) = 0$ for all $ i > 0$ and  $h^i(E^{\bullet})$ is coherent for  $i= -1, 0$.
	\item  A map $\phi: E^{\bullet}  \rightarrow {\LL}_{f} $  in $ \D(\QCoh(\sX))$, 
	where ${\LL}_{f}$ denotes the relative cotangent complex of the stack $\sX$ over $\sY$.
	\item  $h^0(\phi)$ is an isomorphism and $h^{-1}(\phi)$ is surjective. 
\end{enumerate}
\end{defn} 


\begin{defn}\cite[Definition 5.1]{BF}\label{defn:perfectobstructiontheory}
An obstruction theory $E^{\bullet}$ is called a {\sl perfect relative obstruction theory} on $\sX$ with respect to $\sY$ if the complex $E^{\bullet}$ 
is a perfect complex of amplitude contained in $[-1,0]$.  The {\sl relative virtual dimension} of $\sX$ relative to $\sY$
with respect to the perfect relative obstruction theory $E^\bullet$
is defined to be the virtual rank of the perfect complex $E^{\bullet}$ (which is in general a $\Z$-valued locally constant function and an integer when 
$\sX$ is connected). 
\end{defn}

\begin{defn} \label{defn:eq. perfectobstructiontheory}
Let $\tilde{f}: X \to Y$ be a $G$-equivariant map between $G$-schemes. An {\sl equivariant perfect relative obstruction theory} on $X$ relative to $Y$  
is a map $\phi: E^{\bullet} \to \LL_{\tilde{f}}$ in $ \D(\QCoh^G(X))$, the derived category of $G$-equivariant quasicoherent sheaves on $X$,
such that forgetting the $G$-action, the underlying non-equivariant map $\phi: E^{\bullet} \to \LL_{\tilde{f}}$ in $ \D(\QCoh(X))$
is a perfect relative obstruction theory in the sense of \defref{defn:perfectobstructiontheory}.
In the case when $Y = \Spec k$ is the base scheme, $\phi: E^{\bullet} \to \LL_{\tilde{f}}$ is called an equivariant perfect obstruction theory 
on the $G$-scheme $X$. An equivariant perfect relative obstruction theory on $X$ relative to $Y$ is equivalent to 
a perfect relative obstruction theory on the quotient stack $\sX := [X/G]$ relative to the quotient stack $\sY := [Y/G]$ for the
schematic and hence $\DM$-type morphism from $\sX$ to $\sY$ induced by $\tilde{f}$ and 
an equivariant perfect obstruction theory on $X$ is the same as a perfect relative obstruction theory on the quotient stack $\sX := [X/G]$
relative to $\BG$. Note that this is therefore different from a perfect obstruction theory on the stack $\sX$ even
when $\sX$ is a $\DM$-stack.
\end{defn}

\begin{remk}
A perfect complex $E^{\bullet} \in \D(\QCoh(\sX))$ of perfect amplitude contained in $[a, b]$
is said to admit a {\sl global resolution} if it is (globally on $\sX$) quasi-isomorphic 
to a complex of the form $F^{\bullet} = F^a \to \cdots \to F^b$,
where each $F^i$ is a locally free sheaf.
If $\sX$ has the resolution property, i.e
every quasicoherent $\sO_{\sX}$-module is a quotient of a direct sum of locally free sheaves,
then every perfect complex on $\sX$ has a global resolution. 
Moreover, if $F^{\bullet}$ as above is a global resolution of $E^{\bullet}$, then the virtual rank of $E^{\bullet}$ is equal to
$\Sigma_{i=a}^b \rank(F^i)$.
\end{remk}

\begin{defn}\label{defn:virtualtangentbundle}
	Let $f: \sX \to \sY$ be a $\DM$-type morphism of algebraic stacks and let $E^\bullet \to \LL_{{\sX}/\sY}$ be a perfect relative obstruction theory on 
	$\sX$ relative to $\sY$. Suppose $E^\bullet$ and $\LL_{{\sY}/\Spec k}$ admit global resolutions. 
	We define the virtual relative tangent bundle of $\sX$ over $\sY$  
	corresponding to $E^\bullet$ to be
	$$T_{[\sX/\sY, E^\bullet]}^\vir : =[E^{\bullet \vee}] + f^*([\LL_{{\sY}/\Spec k}^\vee]) \in \KK_0(\sX).$$  
\end{defn}

\begin{exm} \label{exm: eq. virtualtangentbundle}
Under the notations of \defref{defn:eq. perfectobstructiontheory}, if $E^\bullet$ and $\LL_{Y}$ admit global  resolutions as bounded complexes
of equivariant locally free sheaves, then the equivariant virtual tangent bundle corresponding to $E^\bullet$ is given by
$$T_{[X/Y, E^\bullet]}^\vir = [E^{\bullet \vee}] + f^*[T_Y] - [\fg_X],$$
where $\fg_X$ denotes the pullback of the Lie algebra of $G$ to $X$.

\end{exm}

For the rest of this section let us fix the following notations. See \cite[Definitions 3.6, 3.10]{BF} for detailed constructions of
the intrinsic normal sheaf and the intrinsic normal cone in the absolute case and see \cite[Definitions 2.12, 2.14]{Manolache} for the relative case.
\begin{setup}\label{setup: Recall}
	Let $f: \sX \to \sY$ be a $\DM$-type morphism of algebraic stacks, where $\sY$ is a stack of pure dimension. 
	For a perfect complex $E^\bullet$, let $h^1/h^0(E^{\bullet \vee})$ denote the stack quotient
	$ h^1/h^0(\tau_{[0,1]} (E^{\bullet \vee}))$. 
	Let $\mathfrak{N}_{\sX/\sY} := h^1/h^0(\LL^\vee_{\sX/\sY})$ denote the intrinsic normal sheaf of $\sX$ over $\sY$
and let the closed subcone stack $\mathfrak{C}_{\sX/\sY} \inj \mathfrak{N}_{\sX/\sY}$ denote the intrinsic normal cone of $\sX$ over $\sY$. 
	Given a perfect relative obstruction theory $E^\bullet  \to \LL_{\sX/\sY}$, there is an associated  vector bundle stack $h^1/h^0(E^{\bullet \vee})$ 
	over $\sX$ such that  $\mathfrak{C}_{\sX/\sY} \inj h^1/h^0(E^{\bullet \vee})$ is a closed immersion (see \cite[Proposition 3.11]{Manolache}).
\end{setup}

\subsubsection{Virtual Structure Sheaf}
 Before we get to the definition of the virtual structure sheaf, we recall some facts related to the $\GG$-theory of vector bundle stacks. 

\begin{lem} \label{Lem:Aff-bun}
	If $\psi: \sM \to \sX$ is an affine bundle, then $\psi^*: G(\sX) \to G(\sM)$ is a homotopy equivalence, and
	in particular when $\sM$ is a vector bundle this establishes the homotopy invariance of $\GG$-theory.
\end{lem}
\begin{proof}
	Associated to the affine bundle $\psi: \sM \to \sX$, there exists a short exact sequence
	of locally free sheaves 
	$$
	0 \to F \to E \to \sO_{\sX} \to 0
	$$
	on $\sX$ such that we have a closed immersion $\P(F) \inj \P(E)$ with complement
	$\sM$. By Quillen's localization theorem and d\'{e}vissage theorem, we have the following localization sequence for $\GG$-theory:
	$$
	G(\sM) \to G(\P(E)) \to G(\P(F)),
	$$
	and now using the projective bundle formula for $G$-theory \cite[Section 3C]{KR18}
	we have the required equivalence.
\end{proof}

\begin{prop}\label{prop:vbundlestack}
	Let $\pi: \mathfrak{F} \to \sX$ be a vector bundle stack on an algebraic stack $\sX$ such that
	$\mathfrak{F}$ has a global presentation as $[E/F]$ for some morphism of vector bundles
	$F \to E$ on $\sX$, then the map $\pi^*: G(\sX) \to G(\mathfrak{F})$ is a homotopy equivalence.
\end{prop}

\begin{proof}
	Since $\mathfrak{F} $ is of the form $[E/F]$ for a morphism of vector bundles
	$d: F \to E$ on $\sX$ for the (additive) action of $F$ on $E$ induced by $d$ (see \cite[Section 1]{BF}),
	the map $\pi$ factors as $\mathfrak{F} \xrightarrow{p} [\sX/F] \rightarrow \sX$,
	where $\sX$ has trivial $F$-action and $p: \mathfrak{F} \to [\sX/F]$ is a vector bundle.
	By homotopy invariance of $G$-theory for vector bundles \lemref{Lem:Aff-bun}, it follows that
	$p^*$ is an equivalence, and therefore it suffices to consider the case when $\mathfrak{F}$
	is the classifying stack of a vector bundle. In this case, the structure map $s: \sX \to \mathfrak{F}$
	is an affine bundle and $\pi \circ s = id_{\sX}$ and so the equivalence again follows from
	\lemref{Lem:Aff-bun}.
\end{proof}

\begin{remk}
If $\sX$ is an algebraic stack with resolution property, 
then every coherent $\sO_{\sX}$-module on $\sX$ is the quotient
of a locally free sheaf. Therefore, by \cite[Proposition 1.4.15]{Deligne}, every vector bundle stack on $\sX$ 
has a global presentation as $[E/F]$, for some morphism of vector bundles
$F \to E$ on $\sX$.
\end{remk}

\begin{defn}\label{defn:virtualstructuresheaf}\cite[\S~2.3, Definition]{YPLee})
Under the notations of  \setref{setup: Recall},
let $\pi: h^1/h^0(E^{\bullet \vee}) \to \sX$ be the vector bundle stack.
The closed embedding $\mathfrak{C}_{\sX/\sY} \inj h^1/h^0(E^{\bullet \vee})$ defines a class $[\sO_{\mathfrak{C}_{\sX/\sY}}]$
in $\GG_0(h^1/h^0(E^{\bullet \vee}))$.
The {\sl virtual structure sheaf} $\sO^{\vir}_{[\sX, E^\bullet]}$ for the perfect (relative) obstruction theory $E^{\bullet}$
is defined to be the unique element in $G_0(\sX)$ such that $\pi^* (\sO^{\vir}_{[\sX, E^\bullet]}) = [\sO_{\mathfrak{C}_{\sX/\sY}}]$,
which is well defined by \propref{prop:vbundlestack}.
If we assume $\sX := [X/G]$ is a quotient stack and $E^\bullet$ is an equivariant perfect (relative) obstruction theory on $X$
(relative to a $G$-scheme $Y$) (see \defref{defn:eq. perfectobstructiontheory}), 
the virtual structure sheaf can be naturally realized as an element of $G_0(G, X)$
and we denote it by $\sO^{\vir}_{[G, X, E^\bullet]}$.
\end{defn}


\subsubsection{Virtual Fundamental Class}
The virtual fundamental class associated to a (relative) perfect obstruction theory on a stack $\sX$ 
($\DM$ over a pure dimensional stack $\sY$) is constructed in \cite{BF}
by intersecting the intrinsic normal cone with the zero section of the vector bundle stack associated to the obstruction theory
to produce a class in $\CH_*(\sX)_{\Q}$. This was done in \cite{BF} in the case when the perfect obstruction theory has a global resolution and 
$\sX$ is of $\DM$-type,
and can be done more generally for any perfect obstruction theory \cite[Section 5.2]{Kresch} 
and for any $\sX$ which admits a stratification by global quotient stacks \cite[Corollary 3.12]{Manolache}.

\begin{defn}\label{defn:virtualfundamentalcclass}
Let notations be as in \setref{setup: Recall} and let $\pi: h^1/h^0(E^{\bullet \vee}) \to \sX$ be the vector bundle stack.
The intrinsic normal cone defines a closed subscheme and therefore a cycle $[\mathfrak{C}_{\sX/\sY}]$ on $h^1/h^0(E^{\bullet \vee})$.
Recall that $\pi^* : \CA_*(\sX) \to \CA_*(h^1/h^0(E^{\bullet \vee}))$ is an isomorphism (see \cite[Theorem 4.3.2]{Kresch}), 
therefore one can define  $[\sX^\vir, E^\bullet] \in \CA_*(\sX)$ as the unique element such that $\pi^*([\sX^\vir, E^\bullet]) = [\mathfrak{C}_{\sX/\sY}]$.

From \cite[Theorem 2.1.12]{Kresch}, it follows for a $G$-scheme $X$ and $\sX = [X/G]$, the associated quotient stack, 
there is a natural map $\psi: \CH_*^G(X) \rightarrow \CA_*(\sX) $ which is an isomorphism. 
We define the virtual fundamental class in $\CH_*^G(X)_{\Q}$ as  $\psi^{-1}([\sX^\vir, E^\bullet])$ and denote it by 
$[\sX^\vir, E^\bullet]$. If $\sY = [Y/G]$ is also a quotient stack for a $G$-scheme $Y$ and $f: \sX \to \sY $ is induced by a $G$-equivariant map
from $X \to Y$, then we denote $[\sX^\vir, E^\bullet]$ by $[G, X^\vir, E^\bullet]$, where now $E^\bullet$ is also used
to notate its corresponding $G$-equivariant perfect obstruction theory.
\end{defn}


\section{Virtual classes and the equivariant Riemann-Roch map}\label{sec:VCRR}
In this section we derive the equivariant virtual Riemann-Roch formula, which for a $G$-scheme $X$ and 
an equivariant perfect (relative) obstruction theory $E^\bullet$ on $X$ gives a
relation between the image of the virtual structure sheaf under the Edidin-Graham Riemann-Roch map $\tau^G_X(\sO^\vir_{[G, X, E^\bullet]})$
and the virtual fundamental class $[G, X^\vir, E^\bullet]$ in the Totaro-Edidin-Graham equivariant Chow group 
$\CH^G_*(X)_\Q$ of $X$.

\subsection{Induced perfect obstruction theory}


The following lemma shows that given an equivariant perfect (relative) obstruction theory
on a $G$-scheme $W$ with finite \'{e}tale stabilizers (relative to a smooth $G$-scheme), it descends to a perfect obstruction theory on the associated
quotient stack $[W/G]$. This is well known to the experts in the area (see \cite[Example 4.14]{Kiem-Park}),
but we spell out the proof here as the explicit construction of the induced perfect obstruction theory on $[W/G]$
is used in the proof of our main theorem. See \defref{defn:eq. perfectobstructiontheory} for the definition of equivariant
perfect obstruction theories.


\begin{lem}\label{lem:inducedpotgeneralsetup}
Let $\sW \xrightarrow{f_1} \sY \xrightarrow{f_2} \sZ$ be a morphism of stacks, 
where $\sZ$ is smooth. 
Further assume that $f_2$ is smooth and  $f_1$, $f_2 \circ f_1$ are of $\DM$-type. 
Then any relative perfect obstruction theory $\phi: F^\bullet  \to \LL_{\sW/\sY}$ defines 
a (non-unique) perfect obstruction theory $\phi': \tilde{E}^\bullet \to \LL_{\sW/\sZ}$ on $\sW$. 
\end{lem}

\begin{proof}

	The  composition of maps $\sW \to \sY \to \sZ$ gives the following distinguished triangle of cotangent complexes in $\D^b(\Coh(\sW))$:
	\begin{equation} \label{eq:cotgt cplx}
	f_1^*\LL_{\sY/\sZ} \to \LL_{\sW/\sZ} \to \LL_{\sW/\sY} \xrightarrow{+1}.
	\end{equation}
	Note that  $\LL_{\sY/\sZ}$ is  concentrated in degree $[0,1]$ as $f_2$ is assumed to be a smooth morphism. 
	Consider the following commutative diagram of distinguished triangles:
	\begin{equation}\label{eq:POTEquivariant}
	\xymatrix
	{
		\tilde{E}^\bullet\ar[r] \ar[d]^{\phi'} & F^\bullet \ar[r]^>>>>>>{a \circ \phi} \ar[d]^{\phi} & f_1^*\LL_{\sY/\sZ}[1] \ar[r]^>>>>>>{+1} \ar[d]^{Id} & \\
		\LL_{\sW/\sZ} \ar[r]  & \LL_{\sW/\sY}  \ar[r]^>>>>>>a  & f_1^*\LL_{\sY/\sZ}[1] \ar[r]^>>>>>>{+1} & ,
	}
	\end{equation}
	where $a: \LL_{\sW/\sY} \to f_1^*\LL_{\sY/\sZ}[1]$ is the map in the distinguished triangle \eqref{eq:cotgt cplx},
	$\tilde{E}^\bullet$ is defined as the object in $\D^b(\Coh(\sW))$ that makes the top horizontal diagram a distinguished triangle 
	and the (non-unique) map $\phi'$ is defined by the axioms for triangulated categories. 
	To prove the lemma we need to show that $\phi' : \tilde{E}^\bullet \to  \LL_{\sW/\sZ}$ defines a perfect obstruction theory, 
	that is we need to show that $h^{-1}(\phi')$ is surjective, $h^0(\phi')$ is an isomorphism 
	and $\tilde{E}^\bullet $ is of perfect amplitude contained in $[-1,0]$.
	
	From \eqref{eq:POTEquivariant} we obtain the following long exact sequence of cohomology sheaves: 
	
	\tiny
	\begin{equation}
	\xymatrix@C-=0.5cm
	{
		h^{-1} (f_1^*L_{\sY/\sZ})	\ar[r] \ar[d]^{\simeq} & h^{-1}(\tilde{E}^\bullet) \ar[r] \ar[d]^{h^{-1}(\phi')} & h^{-1}(F^\bullet) \ar[r] \ar@{->>}[d]^{h^{-1}(\phi)} 
		&  &&&\\
		h^{-1}( f_1^*L_{\sY/\sZ})	\ar[r]  & h^{-1}(L_{\sW/ \sZ}) \ar[r]  & h^{-1}(L_{\sW/\sY}) \ar[r]  &  &&&	\\
		& h^{0} (f_1^*L_{\sY/\sZ})	\ar[r] \ar[d]^\simeq & h^{0}(\tilde{E}^\bullet) \ar[r] \ar[d]^{h^0(\phi')} & h^{0}(F^\bullet) \ar[r] \ar[d]^\simeq & 
		h^{1} (f_1^*L_{\sY/\sZ})  \ar[d]^\simeq &&\\
		&h^{0}( f_1^*L_{\sY/\sZ})	\ar[r] & h^{0}(L_{\sW/ \sZ}) \ar[r]  & h^{0}(L_{\sW/\sY}) \ar[r]  & h^{1}(f_1^*L_{\sY/\sZ})&&
	}
	\end{equation}
	\normalsize
	It now follows from the five lemma applied to the above diagram that $h^0(\phi')$ is an isomorphism and the four lemma establishes
	the surjectivity of $h^{-1}(\phi')$.
	
	To prove that $\tilde{E}^\bullet$ is of perfect amplitude contained in $[-1,0]$ consider the top
	distinguished triangle of \eqref{eq:POTEquivariant}:
	\begin{equation} \label{eq:perfect-ampl}
\tilde{E}^\bullet \rightarrow F^{\bullet} \xrightarrow{a \circ \phi'} f_1^*(L_{\sY/ \sZ})[1] \xrightarrow{+1}.
	\end{equation}
	Since $f_1^* (L_{\sY/\sZ })$ is of perfect amplitude contained in
	$[0,1]$ and $F^{\bullet}$ is of perfect amplitude contained in $[-1,0]$, it suffices to check that $h^1(\tilde{E}^\bullet)$ vanishes.
	From the long exact sequence of cohomology sheaves associated to \eqref{eq:perfect-ampl}, it follows that
	$$h^1(\tilde{E}^\bullet) \simeq h^1(L_{\sY/ \sZ})/ \text{Im} (h^0(F^\bullet) \xrightarrow{h^0(a \circ \phi')} h^1 (L_{\sY/ \sZ})),$$
	since $h^1(F^\bullet)$ vanishes.
	Now by considering the bottom distinguished triangle of \eqref{eq:POTEquivariant} and 
	since $\sW \to \sZ$ is a $\DM$-type morphism  (which implies $h^1(L_{\sW/ \sZ})$
	vanishes), we conclude that $h^0(a)$ is surjective. Further since $h^0(\phi')$ is an isomorphism as proved above it follows that $h^1(\tilde{E}^\bullet)$
	vanishes.
	
\end{proof}

\begin{remk}
 In \lemref{lem:inducedpotgeneralsetup} we need to assume that the morphisms are of $\DM$-type only to conclude that the induced obstruction theory 
 is of tor amplitude contained in $[-1,0]$. 
\end{remk}

The below construction shows how an equivariant perfect (relative) obstruction theory on a $G$-scheme $X$
(relative to a smooth $G$-scheme) can be lifted to a perfect obstruction theory on the scheme $X \times^{G} U$
for a good pair $(V,U)$. Refer to \secref{Sec:Eq-Chow} for a definition of good pairs. 
\begin{const}\label{constr: equivariant}
	Let $\tilde{f} : X \to Y \in \Sch_k^G$ where $Y \in \Sm_k^G$ with $\dim(X) = n$.
	Let $\sX:= [X/G]$ and $\sY := [Y/G]$ and let ${f}: \sX \to \sY$ be the map induced on the quotient stacks. 
	Fix an integer $i$ and let  $(V, U)$ be a be an $l$-dimensional good pair for $n-i$,
	and let $\sV:= [V/G]$ with structure map $q_{\sV}: \sV \to \BG$ and let $\sU := [U/G]$.  
	We have the following distinguished triangle in $D^b(\Coh(\sV))$ given by the composition of maps $\sV \to \BG \to \Spec(k)$: 
	\begin{equation}
	q_\sV^* \LL_{\BG/ \Spec k} \to \LL_{\sV/ \Spec k} \to \LL_{\sV/ \BG} \xrightarrow{+1}.
	\end{equation}
	Note that the cotangent complex of $\sV$, $\LL_{\sV/ \Spec k} \simeq [\Omega_{V/ \Spec k} \to \fg^\vee]$ is concentrated in degrees $0$ and $1$. It follows that $\LL_{\sV/BG}$ is precisely $\Omega_{V/\Spec k} \simeq V$ (with the underlying $G$-action). 
	
	Consider the following diagram of cartesian squares:   

	\begin{equation}\label{eq:cartesianrelative}
	\xymatrix
	{
		\sX \times_\BG \sU \ar[r]^{f_{\sU}} \ar[d]^{p_{\sX}} & \sY \times_\BG \sU \ar[d]^{p_{\sY}} \ar[r]^>>>>>>{\bar{p}_{\sU}} & \sU \ar[d]\\
		\sX     \ar[r]^{f}	                             & \sY         \ar[r]                                                      &       \BG    .       
	}
	\end{equation}
We have natural isomorphisms $\sX \times_\BG \sU \simeq [X \times U/G] \simeq X \times^G U$, where the last two terms are isomorphic as the 
$G$-action on $U$ is free.  
As $p_{\sY}$ is flat, this implies that the left cartesian square of \eqref{eq:cartesianrelative} is Tor-independent.
	We therefore have a natural isomorphism:
	\begin{equation}
		\LL_{\sX \times_\BG \sU/\sY} \xrightarrow{\simeq} p_\sX^* \LL_{\sX/\sY}  \oplus f_{\sU}^*\LL_{\sY \times_{\BG} \sU/\sY}.
	\end{equation}
As the right square of  \eqref{eq:cartesianrelative} is cartesian and all the maps are flat, it follows that 
$\LL_{\sY\times_\BG \sU/\sY} \simeq \bar{p}_{\sU}^* \LL_{\sU/\BG} \simeq \bar{p}_{\sU}^* \Omega_{\sU/\BG}$.  
Therefore, we have $$\LL_{\sX \times_\BG \sU/\sY} \simeq p_\sX^* \LL_{\sX/\sY} \oplus p_\sU^* \Omega_{\sU/ BG},$$
where $p_{\sU}: \sX \times_{\BG} \sU \to \sU$ denotes the projection to $\sU$.

Let $E^\bullet \to \LL_{\sX/\sY}$ be a perfect relative obstruction theory for the map $f: \sX \to \sY$. Then
\begin{equation}\label{eqn:formixed space}
F^\bullet:= p_\sX^* E^\bullet \oplus p_\sU^* \Omega_{\sU/ \BG} \to p_\sX^* \LL_{\sX/\sY} \oplus p_\sU^* \Omega_{\sU/ \BG} \simeq \LL_{(\sX \times_\BG \sU )/\sY}
\end{equation}   
defines a perfect relative obstruction theory on $\sX \times_\BG \sU$ over $\sY$.

Recall that $\sX \times_\BG \sU \simeq X \times^G U$ is an algebraic space and therefore a $\DM$-stack. Now it follows from 
\lemref{lem:inducedpotgeneralsetup} applied to the composition of morphisms
$\sX \times_\BG \sU \to \sY \to \BG$ that
$F^\bullet \to \LL_{(\sX \times_\BG \sU)/\sY}$ induces a perfect obstruction theory:
\begin{equation}\label{eqn:lemmainduced}
\tilde{E}^\bullet \to \LL_{(X \times^G U)/\Spec k},
\end{equation}
such that $\tilde{E}^\bullet$ fits into a distinguished triangle in $D^b(\Coh(X \times^G U))$:
\begin{equation}\label{eqn:lemmainduced2}
p_{\sX}^* \circ f^* (\LL_{\sY/\Spec k}) \to \tilde{E}^\bullet \to F^\bullet \xrightarrow{+1}.
\end{equation}

\end{const}

In the rest of this section we study the relation between the virtual classes 
corresponding to a given equivariant perfect obstruction theory and the induced perfect obstruction theories
constructed above. The following proposition shows that the pullback from $\sX$ to $\sX \times_{\BG} \sU$ of the virtual class associated to a
perfect relative obstruction theory on $\sX$ with respect to $\sY$ is the virtual class associated to the
pullback of the perfect relative obstruction theory (as constructed above).

\begin{prop}\label{prop: pullbackvirtualclass}
With notations as in Construction \ref{constr: equivariant}, let us further assume that $\sY$ is a pure dimensional stack. 
Then we have the following:  
\begin{enumerate}
\item $p_\sX^*(\sO^\vir_{[\sX, E^\bullet]}) = \sO^\vir_{[\sX \times_\BG \sU, F^\bullet]}$ in $\GG_0(\sX \times_\BG \sU)$,

\item $p_{\sX}^* [\sX^\vir, E^\bullet] = [\sX \times_\BG \sU ^\vir, F^\bullet]$ in $A_*(\sX \times_\BG \sU)_\Q$.
\end{enumerate}
 
\end{prop}

\begin{proof}
Consider the composition of maps $\sX \times_\BG \sU \xrightarrow{p_\sX} \sX \xrightarrow{f} \sY$ and let us denote the composite map by $g$. 
Note that all the morphisms are representable, hence in particular they are of $\DM$-type. 
Let $f_{E^\bullet}^!: G_0(\sY) \to G_0(\sX)$ and $g_{F^\bullet}^!: G_0(\sY) \to G_0(\sX \times_\BG \sU)$ 
be the  virtual pull-back maps defined in \cite[Definition 2.2]{QU}. 
By \defref{defn:virtualstructuresheaf} 
(and identifying the Gysin pullback in {\sl loc. it.} with the inverse of $\pi^*$) we have
\begin{equation}
\sO^\vir_{[\sX,  E^\bullet]} = f_{E^\bullet}^!(\sO_{\sY}) ~~\text{and}~~ \sO^\vir_{[\sX \times_{BG} \sU, F^\bullet]}  = g_{F^\bullet}^!(\sO_{\sY}).
\end{equation}
Clearly by construction we have a commutative diagram of distinguished triangles:
\begin{equation}
	\xymatrix
{
 p_{\sX}^*E^\bullet \ar[r] \ar[d] & F^\bullet \ar[r] \ar[d] & p_{\sU}^*\Omega_{\sU/\BG} \ar[r]^>>>>>>{+1} \ar[d]^{\simeq} & \\
  p_{\sX}^*\LL_{f} \ar[r]  & \LL_{g}  \ar[r]  & \LL_{p_{\sX}} \ar[r]^>>>>>>{+1} & 
,}
\end{equation}
which shows that the perfect obstruction theories are compatible and therefore satisfy the hypothesis of \cite[Proposition 2.11]{QU}.
Hence we have
\begin{align}
\sO^\vir_{[\sX \times_{BG} \sU, F^\bullet]} & = g_{F^\bullet}^!(\sO_{\sY}) \\ \nonumber
&= (p_{{\sX}})^!_{ (p_{\sU}^* \Omega_{\sU/\BG})} (f_{E^\bullet}^!(\sO_{\sY}))\\ \nonumber
&= p_\sX^*[\sO^\vir_{[\sX,  E^\bullet]}],  
\end{align}
where the second equality follows from \cite[Proposition 2.11]{QU} and the last equality follows from \cite[Remark 2.3]{QU} as $p_\sX$ is smooth. 
This completes the proof of  $(1)$.

To prove $(2)$ we argue similarly, this time applying the virtual pullback defined for Chow groups. Let $f_{E^\bullet}^!: \CA_*(\sY)_\Q \to \CA_*(\sX)_\Q$ 
and $g_{F^\bullet}^!: \CA_*(\sY)_\Q \to \CA_* (\sX \times_\BG \sU)_\Q$ 
be the virtual pullback maps defined in \cite[Contruction~3.6]{Manolache}. From \cite[Corollary 3.12]{Manolache} it follows that
\begin{equation}
	[\sX^\vir, E^\bullet] = f_{E^\bullet}^! ([\sY]) ~~\text{and} ~~ [\sX \times_\BG \sU^\vir, F^\bullet] = g_{F^\bullet}^!([\sY]).
\end{equation}
Now  we have the following: 
\begin{align}
[\sX \times_\BG \sU^{\vir}, F^\bullet] &= g_{F^\bullet}^!([\sY]) \\        \nonumber
                                            & = (p_{\sX})^!_{ (p_{\sU}^* \Omega_{\sU/\BG})} (f_{E^\bullet}^!([\sY])) \\ \nonumber
                                            & = p_\sX^*([\sX^\vir, E^\bullet]),                                       \nonumber
\end{align}
where the second equality follows from \cite[Theorem 4.8]{Manolache} and the last equality follows as $p_\sX$ is smooth and the perfect
obstruction theory is trivial (see e.g. \cite[Section 5.1]{Kresch} or \cite[Remark 3.10]{Manolache}). 
\end{proof}

Next we show that the virtual classes associated to the given perfect relative obstruction theory on $\sW/\sY$ 
and the descended perfect relative obstruction theory on $\sW/\sZ$
of \lemref{lem:inducedpotgeneralsetup} are the same. 

\begin{lem}\label{lem:equivalenceunderinduced}
	Under the notations of \lemref{lem:inducedpotgeneralsetup}, further assume that $\sY$ and $\sZ$ are pure dimensional. 
	Then we have the following:
	\begin{enumerate}
		\item $\sO^\vir_{[\sW, F^\bullet]} =  \sO^\vir_{[\sW, \tilde{E}^\bullet]} $ in $\GG_0(\sW)$;
		\item $[\sW^\vir, F^\bullet] = [\sW^\vir, \tilde{E}^\bullet]$ in $\CA_*(\sW)_\Q$.
	\end{enumerate}
\end{lem}

\begin{proof}
	The proof of both the claims are similar and $(2)$ is a straight forward generalization of \cite[Proposition 3]{KKP} to the relative case,
	and the proof follows directly using their techniques.  
	For the convenience of the reader we recall the proof in the case of $(1)$. 
	Arguing as in the proof of \cite[Proposition 3.14]{BF} and using \eqref{eq:POTEquivariant} we have the following 
	cartesian square of cone stacks:
	\begin{equation}
	\xymatrix
	{
		\mathfrak{C}_{\sW/\sY} \ar[r] \ar[d] & \mathfrak{C}_{\sW/\sZ } \ar[d] \\
		h^1/h^0(F^{\bullet \vee})      \ar[r]     &  h^1/h^0(\tilde{E}^{\bullet \vee}), 
	}
	\end{equation}
	where the bottom horizontal map is smooth. Since $h^1/h^0(F^{\bullet \vee})$ and $h^1/h^0(\tilde{E}^{\bullet \vee})$ are vector bundle stacks over 
	$\sW$, $(1)$ follows from \defref{defn:virtualstructuresheaf} and the functoriality of flat pull-back for $G$-theory. 
\end{proof}

\begin{remk}
Let $F^\bullet \to \LL_{\sW/\sY}$ and $\tilde{E}^\bullet \to \LL_{\sW/\sZ}$ be as in \lemref{lem:inducedpotgeneralsetup}.
Although $\tilde{E}^\bullet \to \LL_{\sW/\sZ}$ is not uniquely determined by $F^\bullet \to \LL_{\sW/\sY}$, 
\lemref{lem:equivalenceunderinduced} allows us to conclude that the virtual classes $\sO^\vir_{[\sW, \tilde{E}^\bullet]}$ and $[\sW^\vir, \tilde{E}^\bullet]$ are indeed 
uniquely determined by the construction in \lemref{lem:inducedpotgeneralsetup} 
as they coincide with $\sO^\vir_{[\sW, F^\bullet]}$ and $[\sW^\vir, F^\bullet]$ respectively.
\end{remk}

\subsection{Virtual equivariant Riemann-Roch formula} \label{sec:MainTheorem}
In this section we prove equivariant 
versions of virtual Riemann-Roch theorems. Throughout this section we
make the following assumptions regarding $G$ and $X \in \Sch_k^G$ which ensure that the algebraic space $X \times^G U$ associated
to good pairs $(V,U)$ are infact schemes. 

\begin{assmp} \label{assmp}
Let $X \in\Sch^G_k$ such that there is a $G$-equivariant closed immersion $X \inj M$ where $M \in \Sm^G_k$. 
Further assume that one of the following conditions holds:
\begin{enumerate}
\item $M$ is quasi-projective (and hence $M$ and $X$ are $G$-quasi-projective by \cite[Corollary 2.14]{Brion}), or
\item $G$ is a special linear algebraic group, or
\item $G$ is connected. 
\end{enumerate}
\end{assmp}

The fact that for any good pair $(V,U)$, the principal bundle quotient $X \times U \to X \times^G U$ exists in the category of schemes
under assumptions \ref{assmp}
follows from \cite[Proposition 23]{EG}. Recall that a $G$-equivariant map $f: X \to Y$ is said to be
$G$-quasi-projective if there exists a $G$-equivariant locally free $\sO_Y$-module $E$ on $Y$ 
and a quasi-compact $G$-equivariant immersion $X \inj \P (E)$ over $Y$. And $X \in \Sch^G_k$ is said to be $G$-quasi-projective
if $X \to \Spec k$ is a $G$-quasi-projective map.

\begin{thm}\label{thm:main} 
Let $X, Y \in \Sch^G_k$ and $\tilde{f}: X \to Y$ be a $G$-equivariant morphism such that $Y$ is smooth, 
$G$-equivariantly connected and pure dimensional. 
Suppose that there exists a $G$-equivariant closed immersion $X \inj M$ for some $M \in \Sm_k^G$ 
such that the assumptions in \ref{assmp} are satisfied.
    Then for an equivariant perfect relative obstruction theory $E^\bullet \to \LL_{X/Y}$ on $X$ with respect to $Y$ which admits a global resolution:
	\begin{equation}
	\tau^G_{X} (\sO^\vir_{[G, X, E^\bullet]}) = Td^G(T^\vir_{[X/Y, E^\bullet]}) \cap [G, X^\vir, E^\bullet]
	\end{equation}
	in $\CH_*^G(X)_\Q$, where $Td^G$ denotes the equivariant Todd class (cf. \cite[Definition 3.1]{EG1}).
\end{thm}
\begin{proof}

Equivariant Chow groups of $X$ are by definition given by the Chow groups of the quotient scheme $X \times^G U$.
Recall from \secref{para:equivariantRR} that the equivariant Riemann-Roch map is defined by precomposing the  Riemann-Roch map for schemes 
with a flat pullback map on equivariant $G$-theory. 
To prove the theorem we first show that the virtual structure sheaf corresponding to the equivariant perfect relative obstruction theory
pulls back to the virtual structure sheaf corresponding to an absolute perfect obstruction theory on $X \times^G U$.
We then apply \cite[Proposition 3.4 (2)]{FG} to write the image of the virtual structure sheaf on $X \times^G U$ under the Riemann-Roch map
in terms of the virtual fundamental class and the Todd class of the virtual tangent bundle.
Finally we reinterpret these classes in terms of the equivariant classes corresponding to the given equivariant perfect relative obstruction theory.

Let $\sX := [X/G], \sY := [Y/G]$ and let $f: \sX \to \sY$ be the map induced by $\tilde{f}$.  
Let $d$ be the virtual dimension of $X$ with respect to the equivariant perfect relative obstruction theory $E^\bullet$, 
which implies that $[G, X^\vir, E^\bullet] \in \CH_{\dim(Y) + d}^G(X)$.  Let $(V,U)$ be an $l$-dimensional good pair for $\dim(X)-j$,
where $j := \dim(Y) + d$. 
 
 Step 1: Let $s_U: \GG_0(G,X)_\Q \to \GG_0(X \times^G U)_\Q$ be the map defined in \secref{para:equivariantRR}. 
 For $F^\bullet \to \LL_{X \times^G U/\sY}$
  the perfect relative obstruction theory defined in \eqref{eqn:formixed space},
 it follows from \propref{prop: pullbackvirtualclass} that
 \begin{equation}
s_U(\sO^\vir_{[G,X, E^\bullet]}]) = \sO^\vir_{[ X \times^G U, F^\bullet]}
 \end{equation}
 in $\GG_0(X \times^G U)$.
 
 Step 2: Next we compute the Todd class of $X \times^G U$ corresponding to the perfect obstruction theory $\tilde{E}^\bullet$ of 
 \eqref{eqn:lemmainduced} in terms of the equivariant Todd class $Td^G(T^\vir_{[X/Y, E^\bullet]})$. 
 By abuse of notation we let $Td^G(V)$ denote the Todd class $Td(X \times^G (U \times V)) = r^* Td^G(V)$, where
 $r: X \times U \to \Spec(k)$ is the structure map in $\Sch^G_k$.
 
 Consider the following:
 \begin{align*}\label{eqn:Toddcalculation}
 T^\vir_{[X\times^G U, \tilde{E}^\bullet]}
& =^1 Td(\tilde{E}^\bullet) \\ 
& =^2 Td(F^{\bullet \vee}) ~.~Td (p_\sX^* \circ f^* (L_{\sY/\Spec k}^\vee)) \\ 
 & =^3 p_X^*(Td^G(E^{\bullet \vee})) ~.~ p_U^*(Td^G(\Omega_U^\vee)) ~.~ p_X^* ((Td^G(\tilde{f}^* T_Y) ~.~Td^G(\fg_X)^{-1})\\ 
 & =^4 p_X^*(Td^G(T^\vir_{[X, E^\bullet]}) ~.~ Td^G(V),  
 \end{align*}
where $=^1$ follows from \defref{defn:virtualtangentbundle}
 and $=^2$ follows from the multiplicative property of $Td$ and the distinguished triangle 
 \eqref{eqn:lemmainduced2} in $\D^b(\Coh(X \times^G U))$. 
 $=^3$ follows from the definition of $F^\bullet$ (see \eqref{eqn:formixed space}), the definition of the equivariant Todd class $Td^G$
 and since the perfect complex $L_{\sY/\Spec k}^\vee$ on $\sY$ corresponds to the equivariant perfect complex $[\fg_Y \to T_Y]$ on $Y$. 
 The last equality, $=^4$ follows from  \defref{defn:virtualtangentbundle} (see Example \ref{exm: eq. virtualtangentbundle}).
 
 Step 3: Note that by definition $\CH_{j}^G(X) : =  \CH_{j+l-g} (X \times^G U)$. We next claim that:
 \begin{equation} \label{eqn:fundamentalclasschowgroupschange}
 [G, X^{\vir}, E^\bullet] = [X \times^{G} U^\vir, F^\bullet]
 \end{equation}
 in $\CH_{(j+l-g)} (X \times^{G} U)$.
 First note that if the virtual dimension of $X$ with respect to $E^\bullet$ is $d$ then the virtual dimension of $X \times_{G} U$ with respect to 
 $F^\bullet$ is $(l+d)$ (by construction \eqref{eqn:formixed space}), thus both the fundamental classes lie in $\CH_{(j+l-g)}(X \times_{G} U)$
 as $\dim(\sY) = \dim(Y) - g$. From the naturality of the isomorphism $\widehat{A}_* (~) \to A_* (~)$ for quotient stacks 
 (see \cite[Theorem 2.1.12]{Kresch}) it follows that the following diagram commutes, 
 
 \begin{equation}
 \xymatrix
 {
 	\CH^G_j(X) \ar[r]^>>>>>{p_X^*} \ar[d]^\simeq & \CH^G_{j+l} (X \times U) = \CH_{(j+l-g)}(X \times^G U)\ar[d]^\simeq \\
 	{A}_j (\sX)           \ar[r]^>>>>>>>>>>>>{p_\sX^*}          &  A_{j+d} ([(X \times U)/G]),
 }
 \end{equation}
 where the top right equality follows from \cite[Proposition 8(a)]{EG}. Now \eqref{eqn:fundamentalclasschowgroupschange} follows by combining \defref{defn:virtualfundamentalcclass} and \lemref{lem:equivalenceunderinduced}.   
 
 Step 4: By assumption \ref{assmp} the scheme $X$ has a  $G$-equivariant embedding in a smooth scheme $M$.  
This induces an embedding of schemes $X \times^G U \inj M \times^G U$ such that $M \times^G U \in \Sm_k$. 
Thus one can apply \cite[Proposition 3.4 (2)]{FG} to the perfect obstruction theory $\tilde{E}^\bullet$ on the scheme $X \times^G U$ to conclude: 
 \begin{equation}\label{eqn: changeofpot}
 \tau_{X \times^G U}([\sO^\vir_{[X \times^G U, \tilde{E}^\bullet]}]) =Td( T^\vir_{X\times^G U, \tilde{E}^\bullet}) \cap [X\times^G U^\vir, \tilde{E}^\bullet].
 \end{equation}
 Now we have the following:
 \begin{align}
\tau^G_{X} (\sO^\vir_{[G, X, E^\bullet]})  & =^1 Td^G(V)^{-1} \cap \tau_{X\times^G U}(\sO^\vir_{[X \times^G U, F^\bullet]}) \\ \nonumber
                                                               & =^2 Td^G(V)^{-1} \cap \tau_{X\times^G U} (\sO^\vir_{[X \times^G U, \tilde{E}^\bullet]}) \\ \nonumber
                                                               & =^3 Td^G(V)^{-1} ~.~ Td( T^\vir_{[X\times^G U, \tilde{E}^\bullet]}) \cap [X\times^G U^\vir, \tilde{E}^\bullet] 
                                                               \\ \nonumber
                                                               & =^4  
                                                               Td^G(V)^{-1}~.~ Td^G(T^\vir_{[X, E^\bullet]}) ~.~ Td^G(V) 
                                                               \cap  [X\times^G U^\vir, F^\bullet] \\ \nonumber
                                                               & =^5 Td^G(T^\vir_{[X, E^\bullet]}) \cap  [G,X^{\vir}, E^\bullet], \nonumber
 \end{align}
where $=^1$ follows from Step 1 and the definition of the equivariant Riemann-Roch map, 
$=^2$ follows from \lemref{lem:equivalenceunderinduced} $(2)$, 
$=^3$ follows from \eqref{eqn: changeofpot}, 
$=^4$ follows from Step 2 and finally 
$=^5$ follows from Step $3$. 
This completes the proof of the theorem. 
\end{proof}

Let $ch^G$ denote the equivariant chern character map of \cite[Definition 3.1]{EG1}).
Let $\pi: X \to \Spec k \in \Sch^G_k$ be a proper map, i.e. $X$ is a complete $G$-scheme.
Then for an equivariant perfect obstruction theory $E^\bullet \to \LL_{X}$,
the virtual equivariant Euler characteristic of $\alpha \in \KK_0(G,X)$
is defined as
$$
\chi^G_\vir(X, E^\bullet, \alpha) := \chi^G (\alpha \otimes \sO^\vir_{[G, X, E^\bullet]}),
$$
where $\chi^G(\beta)$ is defined as $ch^G(\pi_*(\beta))$ for $\beta \in \GG_0(G,X)$ (see \cite{EG1}).

\thmref{thm:main} allows us to prove  equivariant analogues of the virtual Riemann-Roch Theorems (see \cite[Theorem 3.5]{BF}).

\begin{cor}\label{cor:EVGRR}
Let $X$ be a $G$-scheme satisfying the assumptions in \ref{assmp} and
let $E^\bullet \to \LL_{X}$ be an equivariant perfect obstruction theory on $X$ which admits global resolution
and let $\alpha \in \KK_0(G, X)$.
\begin{enumerate}
\item{(Virtual Equivariant Grothendieck-Riemann-Roch).}
Let  $f: X  \rightarrow Y  \in  \Sch_k^G$ where $f$ is proper and $Y$ is smooth and $G$-quasi-projective.
Then:
		\begin{equation} \label{eq:vGRR}
	ch^G(f_*(\alpha \otimes \sO^\vir_{[G, X, E^\bullet]}) ) ~.~ Td^G(T_Y) \cap [Y] = \\ 
	f_*(ch^G(\alpha) ~.~ Td^G(T_{[X,E^\bullet]}^\vir) \cap [G, X^\vir,E^\bullet])
	\end{equation}
	in $\CH^G_*(X)_\Q$.\\
\item{(Virtual Equivariant Hirzebruch-Riemann-Roch).} If $X$ is a complete $G$-scheme, then:
\begin{equation} \label{eq:vHRR}
\chi^G_\vir(X, E^\bullet, \alpha) = \int_{[G,X^\vir, E^\bullet]} ch^G(\alpha) ~.~ Td^G(T^\vir_{[X, E^\bullet]}).
\end{equation}
\end{enumerate}
\end{cor}

\begin{proof}
The proof of \eqref{eq:vGRR} is identical to \cite[Theorem 3.5]{FG} using 
the properties of the equivariant Riemann-Roch map given in \cite[Theorem 3.1]{EG1}.
By the module property of $\tau^G$ given by Theorem 3.1(c) and the covariance with respect to equivariant proper maps 
of \cite[Theorem 3.1(b)]{EG1}:
$$f_*(ch^G(\alpha) ~.~ Td^G(T_{[X,E^\bullet]}^\vir) \cap [G, X^\vir,E^\bullet]) = \tau^G_Y(f_*(\alpha \otimes \sO^\vir_{[G, X, E^\bullet]}) ).$$
Again by the module property of $\tau^G$ and since $\KK_0(G,Y) \to \GG_0(G,Y)$ is an isomorphism for $Y \in \Sm^G_k$ (cf. \cite[Theorem 5.7]{Thomason}):
$$\tau^G_Y(f_*(\alpha \otimes \sO^\vir_{[G, X, E^\bullet]}) ) = ch^G(f_*(\alpha \otimes \sO^\vir_{[G, X, E^\bullet]} )) ~.~ \tau^G_Y([\sO_Y]).$$
Now for $g: Y \to Spec k$ smooth and $G$-quasi-projective,  using \cite[Theorem 3.1(d)(ii)]{EG1} we have
\begin{align*}
\tau^G_Y([\sO_Y]) &= Td^G(T_Y) \cap g^* \tau^G_{\Spec k}([\sO_{\Spec k}]) \\
                          	  & = Td^G(T_Y) \cap [Y].
\end{align*}
This proves (1). (2) is a special case of (1) for the morphism $X \to \Spec k$.
\end{proof}

\section{Virtual structure sheaf and the Atiyah-Segal map} \label{Sec:Atiyah-Segal}
In this section we work over the base field $\C$
and we consider algebraic $K$-theory and Chow groups with complex coefficients;
but for simplicity of notation we shorten $\GG_0(G,X) \otimes_\Z \C$ as $\GG_0(G,X)$
and $\CH_*^G(X)  \otimes_\Z \C$ as $\CH_*^G(X)$.
We recall the construction of the Atiyah-Segal map in Section \ref{sec:Atiyah-Segal-map} and study the 
image of the virtual structure sheaf under this map for a given equivariant perfect obstruction theory in Section \ref{sec:AS-Virtual sheaf}.

\subsection{Atiyah-Segal isomorphism} \label{sec:Atiyah-Segal-map}
Let $X \in \Sch_\C^G$ such that $G$ acts properly on $X$.  
For $X$ smooth the Atiyah-Segal map was defined by Edidin-Graham as a part of \cite[Theorem 6.7]{EG2} 
and for general $X$ in \cite[Theorem 6.5]{EG4}. 
This was extended to higher $K$-theory in \cite{KS20}. 
We recall here briefly the construction of the map for $\GG_0$ from the works of Edidin-Graham \cite{EG2} and \cite{EG4}
without getting into the technical details. We refer to their papers for proof of the non-trivial statements which are stated as facts below,
but are in fact quite involved.
The purpose of this section is to set notations and give an outline of the
definitions which would aid in the proof of   \thmref{thm:Virtual-sh-AS}.

For a linear algebraic group $G$, 
let $R(G)$ denote the representation ring of $G$, 
which in our notation is $\KK_0(G, \Spec \C)$.
There is a bijection between the semi-simple conjugacy classes in $G$ and the
maximal ideals in $R(G)$ {\color{black} (see \cite[Proposition 2.5]{EG2}) }.
Given a semi-simple conjugacy class $\psi_i$ we denote by $\fm_{\psi_i}$ the corresponding maximal ideal in $R(G)$ under this bijection. 
The maximal ideal $\fm_1 \subset R(G)$,
defined as the kernel of the rank map $R(G)  \surj \C$, is called the augmentation ideal of $G$.

Recall that $G$ is said to act properly on $X$, if the map $(m, pr_X): G \times X \rightarrow X \times X$ is proper,
where $m: G \times X \to X$ defines the group action on $X$.  
Let $X$ be a $\C$-scheme with proper $G$-action.
{\bl The algebraic analogue of the Atiyah-Segal theorem from topology
provides an isomorphism between the  localization of the equivariant
algebraic $K$-theory of $X$ at a given maximal ideal  with the localization of a corresponding component of the  inertia scheme $I_X$ at the augmentation ideal}.
We remark that in the case of proper actions, the equivariant $K$-theory is supported at a finite number of maximal ideals of $R(G)$,
therefore the localization and completion of $\GG_0(G,X)$ at maximal ideals coincide, which justifies the 
comparison of the map at the level localizations of the $K$-theory to the classical topological Atiyah-Segal map.
Recall that the inertia scheme $I_X$ is the $G$-scheme defined as the pullback in the following square of $G$-schemes: 
\begin{equation}
	\xymatrix
	{
	I_X \ar[r] \ar[d] & X \ar[d]^{\Delta} \\
	G \times X \ar[r]_{(m, pr_X)} & X \times X.
}
\end{equation}
The closed points of $I_X$ correspond to pairs $(g,x)$ such that $g \cdot x = x$. 

Since the $G$-action on $X$ is proper,
there is a finite set of
semi-simple conjugacy classes 
$\Sigma^G_X = \{\psi_1,\cdots, \psi_n\}$
in $G$ such that the fixed point locus $X^g \subset X$ is non-empty
if and only if $g \in \psi_i$ for some $\psi_i \in \Sigma^G_X$.
Moreover, the natural map given by localizations at maximal ideals
$$\GG_0(G,X) \xrightarrow{\simeq} \bigoplus_{\psi \in \Sigma^G_X} \GG_0(G,X)_{\fm_\psi}$$
is an isomorphism of $R(G)$-modules {\bl (see \cite[Proposition 3.6]{EG2})}.
Also, in this case the inertia scheme has a $G$-equivariant decomposition as $I_X = \coprod_{\psi \in \Sigma^G_X} I^{\psi}_X$, 
where $I^\psi_X$ corresponds to the set of closed points $(g,x)$ of $I_X$ such that $g \in \psi$. 

\subsubsection{The Atiyah-Segal map} \label{sec:AS-map}
{\bl In this section we recall the definition of the Atiyah-Segal map from \cite[Theorem 6.5]{EG4} closely following the presentation of \cite[\S~6.3]{KS20}}. Let $g \in G$ be such that $X^g \neq \emptyset$ and let $\psi \in \Sigma^G_X$ denote the corresponding semi-simple conjugacy class. 
Let $Z_g$ denote centralizer of $g$ in $G$. Let $\fm_g \subset R(Z_g)$ denote the maximal ideal of $Z_g$ corresponding
to the conjugacy class $\psi \cap Z_g = \{g\}$ of $Z_g$.
Let
$\vartheta^{\psi}_X: \GG_0(G,X)_{\fm_\psi} \to \GG_0(G, I^\psi_X)_{\fm_1}$ be defined as the composite map:
\begin{equation}\label{eqn:Atiyahsegal}
	\GG_0(G,X)_{\fm_\psi} \xrightarrow{(i^g_!)^{-1}} \GG_0(Z_g, X^g)_{\fm_g} \xrightarrow{t_{g}^{-1}} \GG_0(Z_g, X^g)_{\fm_1} 
	\xrightarrow{\mu^g_*} \GG_0(G,I^\psi_X)_{\fm_1},  
\end{equation}
where the maps are as defined below. 

 {\it The map $i^g_!$}.  
	As $i^g: X^g \inj X$ is a $Z_g$-equivariant closed immersion, there is a natural pushforward map 
	$i^g_* : \GG_0(Z_g, X^g) \to \GG_0(Z_g, X)$ and it follows from \cite[Theorem 5.1]{EG4} that 
	$i^g_*$ is an isomorphism after localization at $\fm_g$. 
	The restriction of groups defines an exact functor from $\Coh^G(X) \to \Coh^{Z_g}(X)$, which induces a map
	$\res^{G}_{Z_g}: \GG_0(G,X) \to \GG_0(Z_g, X)$ and it  follows from \cite[Corollary 5.5(b)]{EG4} that on localizations,
	the restriction map 
	$\res^{G}_{Z_g}: \GG_0(G, X )_{\fm_{\psi}} \to \GG_0(\sZ_g, X)_{\fm_g}$ is again an isomorphism. 
	One defines $i^g_!$ as the composition $(\res^G_{Z_g})^{-1} \circ i^g_*$.  
	
 {\it The twisting map $t_g$}. 
 Let $P = <g>$ denote the central cyclic subgroup of $Z_g$ generated by $g$. 
As $G$ acts properly on $X$, $g$ is an element of finite order. 
Then $P$ is a finite abelian group and is therefore semi-simple, let $\widehat{P}$ denote the group of characters of $P$. 
Now as $g$ acts trivially on $X^g$ it follows that for any $Z_g$-equivariant coherent sheaf $\sF$ on $X^g$,
we have an eigensheaf decomposition $\sF = \bigoplus_{\chi \in \widehat{P} } \sF_{\chi}$, 
where each $\sF_{\chi} \in \Coh^{Z_g}(X^g)$.
{\bl Now following \cite[\S~6.2]{EG2}, one defines an action of $P$ on $\GG_0(Z_g, X^g)$ as follows.} For any $\sF \in \Coh^{Z_g}(X^g)$,
let $[\sF] \in \GG_0(Z_g, X^g)$ denote its corresponding class. 
For each $a \in P$ one defines $t_a([\sF]) = \bigoplus_{\chi \in \widehat{P} } \chi(a^{-1}) \sF_{\chi}$, 
which extends linearly to define an automorphism of $\GG_0(Z_g, X^g)$.  
On $R(Z_g)$, $t_{g}^{-1}(\fm_g) = \fm_1$, hence it follows that on restricting to localizations we have an isomorphism
$t_{g}^{-1}: \GG_0(Z_g, X^g)_{\fm_g} \rightarrow \GG_0(Z_g, X^g)_{\fm_1}$.

 {\it Morita isomorphism $\mu^g_*$}. 
The map $\mu^g_*: \GG_0(\sZ_g, X^g)_{\fm_1} \to \GG_0(G, I^\psi_X)_{\fm_1}$ is the Morita isomorphism discussed below in 
Example \ref{exm:Morita exm} and  since
$\mu^g_*$ is $R(G)$-linear (see \cite[Remark 3.2]{EG2}) where the $R(G)$-module structure on
$\GG_0(Z_g, X^g)$ is given by the restriction $R(G) \to R(Z_g)$,
it follows that it induces an isomorphism $\GG_0(Z_g, X^g)_{\fm_1} \to \GG_0(G, I^\psi_X)_{\fm_1}$ 
(where the localization on the left is with respect to the augmentation ideal in $R(Z_g)$ 
and the localization on the right is with respect to the augmentation ideal in $R(G)$). 

Note that all the maps in \eqref{eqn:Atiyahsegal} are isomorphisms and therefore the map 
$\vartheta^{\psi}_X$ is an isomorphism for each $\psi \in \Sigma^G_X$. 

\begin{remk} \label{remk:AS-map-indep-of-g}
{\bl Although apriori the definition of the map $\vartheta^{\psi}_X$ depends on the choice of a $g \in \psi$,
it follows from { \cite[Lemma 6.5]{EG2}} (or \cite[Lemma 6.7]{KS20}) that the map is indeed independent of this choice, i.e., 
the composite maps }
$\mu^g_* \circ t_{g}^{-1} \circ (i^g_!)^{-1}$ and 
$\mu^h_* \circ t_{h}^{-1} \circ (i^h_!)^{-1}$ 
coincide if $g, h \in \psi$. This is  proved  by showing that the following diagram commutes:
\begin{equation}\label{eqn:Indep-0}
\xymatrix@C1pc{
\GG_0(G,X)_{\fm_{\psi}} \ar[r]^-{(i^g_!)^{-1}} \ar[d]_{Id} &  
\GG_0(Z_g,X^g)_{\fm_g} \ar[r]^-{t_{g}^{-1}} \ar[d]^{\theta_{X_*}} & \GG_0(Z_g,X^g)_{\fm_1} 
\ar[r]^-{\mu^g_*} \ar[d]^{\theta_{X_*}} &  \GG_0(G,I^{\psi}_X)_{\fm_1} \ar[d]^{Id}  \\
\GG_0(G,X)_{\fm_{\psi}} \ar[r]^-{(i^h_!)^{-1} } & \GG_0(Z_h,X^h)_{\fm_h} 
\ar[r]^-{t_{h}^{-1}} & \GG_0(Z_h,X^h)_{\fm_1} \ar[r]^-{\mu^h_*} & 
\GG_0(G,I^{\psi}_X)_{\fm_1},}
\end{equation}
where $\theta_X$ is an isomorphism between the stacks
$[X^g/Z_g]$ and $[X^h/Z_h]$ 
which is constructed as follows.
Since $g,h$ are in the same conjugacy class, 
there exists a $l \in G$ such that $h = lgl^{-1}$. We have isomorphisms
between the centralizers $Z_g$ and $Z_h$ given by
$\phi: Z_g \xrightarrow{\simeq} Z_h ; ~~ \phi (z) = lzl^{-1}$
and between the fixed point loci $u: X^g \to X^h$
given by $x \mapsto lx$. 
If we let $Z_g$ act on $X^h$ via the canonical action of $Z_h$ 
composed with $\phi$, then $u$ is $Z_g$-equivariant.
Therefore we have an isomorphism of stacks $\theta_X: [X^g/Z_g] \to [X^h/Z_h]$
defined as the composite
\begin{equation}\label{diag:defnofphi_*0} 
[X^g/Z_g] \xrightarrow[u]{\simeq} [X^h/Z_g] \xrightarrow[\phi]{\simeq} [X^h/Z_h].
\end{equation}

\end{remk}

\begin{defn}\label{defn:Atiyah-Segal-map}
Let $G$ act properly on $X$.
The Atiyah-Segal map 
$\vartheta^G_X: \GG_0(G,X) \to \GG_0(G,I_X)_{\fm_1}$
is defined to be the composite isomorphism
\begin{equation} \label{eqn:Atiyahsegal-2}
\GG_0(G,X) \xrightarrow{\simeq} \oplus_{\psi \in \Sigma^G_X} 
\GG_0(G, X)_{\fm_{\psi}} 
\xrightarrow{\oplus {\vartheta^{\psi}_X}}
\oplus_{\psi \in \Sigma^G_X} \GG_0(G,I^{\psi}_X)_{\fm_1} 
\xrightarrow{\simeq} \GG_0(G,I_X)_{\fm_1}.
\end{equation}
\end{defn}

\subsection{Morita equivalence}\label{sec:Morita}
Let us begin by briefly recalling the Morita isomorphism in our context. 
Note that this works in general over any base.
Let $G$ be a linear algebraic group and let $H \subset G$ be a closed subgroup. 
Let $W$ be an $H$-scheme,  one can consider $ G \times W$ as a $G \times H$-scheme, 
where the action is given by $((g,h),(g',w)) \mapsto (gg'h^{-1}, hw)$ where $(g,h) \in G \times H$ 
and $(g', w) \in G \times W$.  
Both the $G$-action and the $H$-action on $G \times W$ are free. 
Let $Y := G \times^H W$ denote the geometric quotient under the $H$-action (which exists as an algebraic space). 
We briefly recall the  Morita isomorphism which provides an isomorphism of  stacks 
$\mu: [W/H] \xrightarrow{\simeq} [Y/G]$.  Consider the following, 

\begin{equation}
\xymatrix
{
W  & G \times W \ar[r]^{\tilde{\pi}_H} \ar[l]_{\tilde{\pi}_G}  &  Y ,
}
\end{equation}
where $\tilde{\pi}_G$ and $\tilde{\pi}_H$ denote the quotient by the actions of $G \times 1$ and $1 \times H$, respectively. 
As the actions are free it follows that $\tilde{\pi_G}$ is a principal $G$-bundle and $\tilde{\pi}_{H}$ is a principal $H$-bundle. 
It is easy to see that this induces an isomorphism of stacks (see for e.g. \cite[Lemma 2.3]{KS20}):
\begin{equation}\label{eq: moritadefinition}
	[W/H]  \xleftarrow[\pi_G]{\simeq} [G \times W/ G \times H] \xrightarrow[\pi_H]{\simeq} [Y/G].
\end{equation} 
We let $\mu = \pi_H \circ \pi_G^{-1}: [W/H] \to  [Y/G]$ denote the Morita isomorphism.

\begin{prop}\label{prop:Moritainducedperfectobstructiontheory}
	Let notations be as above and let 
 $\phi: F^\bullet \to \LL_{W/k}$ be an $H$-equivariant perfect obstruction theory on $W$. 
 Then we can naturally define a $G$-equivariant perfect obstruction theory,  $\tilde{E}^\bullet \to \LL_{Y/k}$ on $Y$. 
 Further, under the Morita isomorphism $\mu_*: \GG_0(H, W)  \to  \GG_0(G,Y)$, 
 we have $\mu_*(\sO^\vir_{[H, W, F^\bullet]}) = \sO^\vir_{[G, Y, \tilde{E}^\bullet]}$.
\end{prop}

\begin{proof}
We first construct a $G$-equivariant perfect obstruction theory on $Y$.
Consider the following commutative diagram of Tor-independent, Cartesian squares:
\begin{equation}\label{eqn:Morita1}
	\xymatrix
	{
	[G \times W/G \times H] \ar[r]^-{p_W} \ar[d]_{p_G}  \ar@/^2pc/[rr]^{\pi_G}         \ar[rd]_{t_1}     & [W/ G \times H] \ar[r]^-{j_1} \ar[d]^{s} & 
	[W/H] \ar[d]^{t_2}  \\
	[G/ G \times H]                \ar[r]_-{i_1}                    &  \BGH              \ar[r]_-{q_H}                       & \BH, 
}
\end{equation}
where $p_W$, $p_G$ and $q_H$ are induced by the projection maps from $G \times W$ to $W$ and $G$ and $G \times H$ to $H$, respectively. 
The map $i_1$ is induced by the structure map $G \to \Spec k$, where $G \times H$ acts on $G$ by $(g,h).g' = gg'h^{-1}$.
The maps
$s$, $t_1$ and $t_2$ are also induced by the structure maps to $\Spec k$ and $j_1$ is induced by the  quotient map
as $W$ has a trivial $G$-action.

We have a $G \times H$-equivariant perfect obstruction theory on $G \times W$ defined by
$$\psi: F_1^\bullet :=   p_G^* \LL_{i_1}  \oplus \pi_G^* F^\bullet \xrightarrow{(id,\pi_G^*(\phi))}  
                                                                        p_G^* \LL_{i_1} \oplus p_W^* \LL_s \simeq \LL_{t_1}.$$ 

Now consider the following commutative diagram:

\begin{equation}\label{eqn:Morita2}
	\xymatrix
	{
	[G \times W/ G \times H] \ar[r]^-{\pi_H} \ar[d]_{t_1} & [Y/G]  \ar[d]^{t_3}\\
	\BGH  \ar[r]_-{q_G}                                  & \BG,
}
\end{equation} 
where $t_3$ is induced by the structure map to $\Spec k$ and $q_G$ denotes the map from $\BGH \to \BG$ induced by the 
group homomorphism $G \times H \to G$ given by projection
(which is not in general representable). 
Since  $\pi_H$ is an isomorphism, the composite map $t_3 \circ \pi_H$ is  representable. 
Applying \lemref{lem:inducedpotgeneralsetup} to the composition of maps $[G \times W/ G \times H] \to \BGH \to \BG$, 
we obtain a a perfect obstruction theory $G_1^\bullet \to \LL_{(t_3 \circ \pi_H)}$. 
Since $\pi_H^*$ is an isomorphism, we get a perfect obstruction theory 
$\tilde{E}^\bullet \to \LL_{t_3}$ which under $\pi_H^*$ gives $G_1^\bullet \to \LL_{(t_3 \circ \pi_H)}$.
Now $\tilde{E}^\bullet \to \LL_{t_3}$ is the required $G$-equivariant perfect obstruction theory on $Y$. 

To complete the proof  we need to show that $\mu_*$ respects the virtual sheaves, which reduces to showing $\pi_H^*(\sO^\vir_{[[W/H], F^\bullet]}) = \pi_G^*(\sO^\vir_{[[Y/G], \tilde{E}^\bullet]})$ 
in $\GG_0(G \times H, G \times W)$. 
Consider the following commutative diagram obtained from \eqref{eqn:Morita1}:

\begin{equation} \label{eqn:Morita3}
\xymatrix
{
[G \times W/ G \times H] \ar[r]^-{\pi_G} \ar[d]_{t_1} & [W/H] \ar[d]^{t_2} \\
\BGH \ar[r]_-{q_H} & \BH,.
}
\end{equation}
Since $\pi_G$ is an isomorphism, as before $F^\bullet \to \LL_{t_2}$ 
defines a relative perfect obstruction theory  
$\pi_G^*(\phi) : \pi_G^*F^\bullet \to \LL_{t_2 \circ \pi_G}$. 
The composite map
$[G \times W / G \times H] \xrightarrow{t_1} \BGH \xrightarrow{q_H} \BH $
induces a distinguished triangle of cotangent complexes in $\D^b(\Coh([G \times W/ G \times H]))$: 
$$
\LL_{q_H \circ t_1} \rightarrow \LL_{t_1} \xrightarrow{+1} t_1^* \LL_{q_H} [1] \rightarrow,
$$
and the map $\LL_{t_1} \xrightarrow{+1} t_1^* \LL_{q_H} [1]$ is split by the natural map of contangent complexes
$t_1^* \LL_{q_H} [1] = p_G^* i_1^* \LL_{q_H} [1] \xrightarrow{u^{-1}} p_G^*\LL_{i_1} \rightarrow \LL_{i_1 \circ p_G} = \LL_{t_1}$,
where $p_G^*\LL_{i_1} \xrightarrow[u]{+1} p_G^* i_1^* \LL_{q_H} [1]$ is an isomorphism since 
$q_H \circ i_1$ is an isomorphism.
Therefore we have a commutative  diagram of distinguished triangles on  $[G \times W/G \times H]$:
\begin{equation}
\xymatrix{
\pi_G^*F^\bullet   \ar[r] \ar[d]^{\pi_G^*(\phi)} & F^\bullet_1 \ar[r] \ar[d]^{\psi}  & p_G^* \LL_{i_1} \ar[r] \ar[d]^{u} & \\
\LL_{q_H \circ t_1}     \ar[r]    &  \LL_{t_1}   \ar[r] & p_G^* i_1^* \LL_{q_H} [1]    \ar[r]&.
}
\end{equation}
It follows from \lemref{lem:equivalenceunderinduced} that
\begin{equation}\label{eq:Morita4}
\sO^\vir_{[[G \times W/ G \times H], F_1^\bullet]} = \sO^\vir_{[[G \times W/ G \times H], \pi_G^*F^\bullet]}
\end{equation}
in $\GG_0([G \times W/ G \times H])$.  Next consider the composite map $[G \times W / G \times H] \xrightarrow{\pi_G} [W/H] \xrightarrow{t_3} \BH$. 
As $\pi_G$ is an isomorphism, it follows from definitions that 
$\sO^\vir_{[[G \times W/ G \times H], \pi_G^*F^\bullet]} = \pi_G^* (\sO^\vir_{[[W/H], F^\bullet]})$
in $\GG_0([G \times W/ G \times H])$.  Therefore from \eqref{eq:Morita4} we can conclude that:
\begin{equation}\label{eqn:Morita5}
\pi_G^* (\sO^\vir_{[[W/H], F^\bullet]} )= \sO^\vir_{[[G \times W/ G \times H], F_1^\bullet]}.
\end{equation}
As $G_1^\bullet \to \LL_{t_2 \circ \pi_H}$ is induced from  \lemref{lem:inducedpotgeneralsetup} applied to the perfect obstruction theory 
$F_1^\bullet \to \LL_{t_1}$ it follows from \lemref{lem:equivalenceunderinduced} that:
\begin{equation}\label{eqn:Morita6}
\sO^\vir_{[[Y/G], G_1^\bullet]} = \sO^\vir_{[[G \times W/ G \times H], F_1^\bullet]}. 
\end{equation}
It then follows immediately that 
\begin{align*}
\pi_H^* (\sO^\vir_{[[Y/G], \tilde{E}^\bullet]} ) &= \sO^\vir_{[G \times W/ G \times H], \pi_H^*\tilde{E}^\bullet]} \\
                        &=  \sO^\vir_{[G \times W/ G \times H], G_1^\bullet]}.                         
 \end{align*}
The proposition now follows by comparing the above equality with \eqref{eqn:Morita6} and \eqref{eqn:Morita5}.
\end{proof}

\begin{exm} \cite[Section 4.1]{EG2}, \cite[Section 5.4]{EG4} \label{exm:Morita exm}
Let $g \in G$ be a semi-simple element of $G$ and let us denote its conjugacy class by $\psi$ and let the centralizer of $g$ in $G$ be denoted by $Z_g$. 
Let $X \in \Sch^G_{\C}$ be a scheme with proper action.
Then the natural map $G \times X^g \to I^\psi_X$ is a principal $Z_g$-bundle 
which induces an isomorphism of stacks $[I^\psi_X/G] \cong [G \times^{Z_g} X^g /G]$. 
By \eqref{eq: moritadefinition} we have an induced Morita isomorphism $\mu^g: [X^g/Z_g] \to [I^\psi_X/G] $
\end{exm}

\subsection{Virtual structure sheaf on the inertia scheme} \label{sec:AS-Virtual sheaf}
Given an equivariant perfect obstruction theory on a $G$-scheme $X$ with a proper $G$-action,
the fixed point locus $X^g$ of an element $g \in G$ carries an associated $g$-fixed perfect obstruction theory,
which is equivariant with respect to the action of the centralizer subgroup $Z_g$ of $g$,
and this is compatible with the given perfect obstruction theory on $X$.
For the case of $\C^*$-actions, this was observed by Graber-Pandharipande in \cite{GP99} and we note
in the following lemma that their proof also works for other groups.
And we use this to induce a compatible $G$-equivariant perfect obstruction theory 
on the inertia scheme $I_X$ under the Morita equivalence. 

\begin{lem} \cite[Proposition 1]{GP99}\label{lem:GPInertia}
Let $X \in \Sch_{\C}^G$ such that the $G$-action is proper. Suppose that there exists a $G$-equivariant closed immersion $X \inj M$ for some $M \in \Sm_\C^G$. 
Let $E^\bullet \to \LL_{X/\C}$ be a $G$-equivariant perfect obstruction theory which admits a global resolution. Then we have a canonical induced $G$-equivariant perfect obstruction theory on $I_X$. 
\end{lem}

\begin{proof}
Let notations be as in Section \ref{sec:Atiyah-Segal-map}. Let $g \in \psi$ such that $X^g \neq \emptyset$. 
The first step is to construct a canonical $Z_g$-equivariant perfect obstruction theory on $X^g$. Let $P = <g>$ 
denote the finite cyclic subgroup of $G$ generated by $g$. 
Now as $E^\bullet \to \LL_{X/\C}$ is a $G$-equivariant perfect obstruction theory on $X$, it can also be viewed, by restriction of groups, as 
$Z_g$-equivariant and $P$-equivariant perfect obstruction theories on $X$. 
Let $i^g: X^g \inj X$ denote the $Z_g$-equivariant closed immersion. 
It follows from the proof of \cite[Proposition 1]{GP99}, that $F_g^\bullet := (i^{g,*} E^\bullet)^{\mathrm{fix}} \to \LL_{X^g/\C}$ 
defines a $Z_g$-equivariant perfect obstruction theory on $X^g$, 
where $(-)^{\mathrm{fix}}$ denotes the invariant (or fixed) part under the eigensheaf decomposition on $X^g$ with respect to the $P$-action.  
Recall from Example \ref{exm:Morita exm} that $I^\psi_X \simeq G \times^{Z_g} X^g$ as $G$-schemes, 
therefore it follows from \propref{prop:Moritainducedperfectobstructiontheory} that 
$F^\bullet_g \to \LL_{X^g/\C}$ induces a $G$-equivariant perfect obstruction theory on 
$I^\psi_X$, which we denote by $\tilde{E}^\bullet_\psi$. 
As $I_X $ is a finite disjoint union of $I^\psi_X$ it follows that $\tilde{E}^\bullet := \oplus_{\psi\in \Sigma^G_X} \tilde{E}^\bullet_\psi \to \LL_{I_X/\C}$ is 
a $G$-equivariant perfect obstruction theory on $I_X$ .
\end{proof}

For a scheme with a $\C^*$-action, the $K$-theoretic virtual localization formula was proved by Qu in \cite[Theorem 3.3]{QU}. 
For a scheme $X$ with $G$-action and $g \in G$ such that $X^g \neq \emptyset$, 
one can consider the analogous virtual localization statement with respect to the $g$-fixed locus $X^g$
if we restrict to the $Z_g$-action. 
Let us briefly recall the set-up as in {\sl loc. cit}.
Given a $G$-equivariant perfect obstruction theory $E^\bullet$ on $X$,
the virtual conormal bundle to the $Z_g$-equivariant closed immersion $i^g: X^g \inj X$ is defined as
$N_{[i^g, E^\bullet]}^{\vir^\vee} := (i^{g,*}E^\bullet)^{\mathrm{mov}}$, where as discussed before 
under the eigensheaf decomposition on $X^g$ with respect to the $P$-action, we have
a decomposition 
$$ i^{g,*}E^\bullet = (i^{g,*}E^\bullet)^{\mathrm{fix}} \oplus (i^{g,*}E^\bullet)^{\mathrm{mov}}$$
into fixed and moving parts.
As we assume that $E^\bullet$ admits a global resolution, in particular this implies that $N_{[i^g, E^\bullet]}^{\vir^\vee}$ has a global resolution.
Also note that the map of $R(Z_g)$-modules given by 
$$\cap ~\Lambda_{-1}([N_{[i^g, E^\bullet]}^{\vir^\vee}]): \GG_0(Z_g,X^g) \to \GG_0(Z_g,X^g)$$
is invertible after localizing at $\fm_g$. This would essentially follow from the proof of \cite[Theorem 3.3(b)]{EG2}
by using the fact that $N_{[i^g, E^\bullet]}^{\vir^\vee}$ is a complex of sheaves, such that the
characters in their eigensheaf decompositions with respect to the $P$-action are all non-trivial.

\begin{lem} \label{lem:vir-localization}
	With notations as above, under the proper pushforward $i^g_*: \GG_0(Z_g, X^g)_{\fm_g} \to \GG_0(Z_g, X)_{\fm_g}$:
	\begin{equation}
	i^g_* \left(\frac{\sO^\vir_{[Z_g, X^g,F_g^\bullet]}}{\Lambda_{-1}([N_{[i^g, E^\bullet]}^{\vir^\vee}])}\right) = \sO^\vir_{[Z_g, X, E^\bullet]}  
	\end{equation}  
	in $\GG_0(Z_g, X)_{\fm_g}$.
\end{lem}

\begin{proof}
The proof  is identical to the proof of \cite[Theorem 3.3]{QU} once we note that 
$(i^g)_*: \GG_0(Z_g, X^g)_{\fm_g} \to \GG_0(Z_g, X)_{\fm_g}$ is an isomorphism by \cite[Theorem 5.1]{EG4}. 
\end{proof}

The following lemma shows that the equivariant perfect obstruction theory
induced by restriction of groups produces compatible virtual structure sheaves
under the restriction functor on $K$-theories.

\begin{lem} \label{lem:restriction-vir-sh}
Let $H \inj G$ be a closed subgroup and let $W \in \Sch_k^G$. Let  $E^\bullet \to \LL_{W/k}$ be a 
$G$-equivariant perfect obstruction theory on $W$. 
Then $\mathrm{res}^G_H(\sO^{\vir}_{[G, W, E^\bullet]}) = \sO^{\vir}_{[H, W, E^\bullet]}$ in $\GG_0(H,W)$.
\end{lem}
\begin{proof}
As  $H \inj G$ is a  closed subgroup of a smooth affine group scheme $G$, 
the induced morphism $f: \BH \to \BG$ is smooth (since fppf locally on $\BG$, $\BH \times_{\BG} \Spec k = G/H$ is smooth).
Consider the following Cartesian square:
$$
\xymatrix{
	[W/H] \ar[r]^{g} \ar[d]_{q} & [W/G] \ar[d]^{p} \\
	\BH \ar[r]^{f} & \BG.}
$$
Let us denote by $E^\bullet \to \LL_p$ the relative perfect obstruction theory (corresponding to the $G$-equivariant perfect obstruction theory on $W$),
it induces a relative perfect obstruction $g^*E^{\bullet} \to \LL_q$  (the $H$-equivariant perfect obstruction theory on $X$ obtained
by restricting to the $H$-action).
The restriction functor $\QCoh^G(W) \to \QCoh^H(W)$ is given by $g^*: \QCoh([W/G]) \to \QCoh([W/H])$ under the
identification of equivariant quasi-coherent sheaves and quasi-coherent sheaves on quotient stacks (see \cite[Section 2.1]{AOV}).
Under this identification,
\begin{align*}
\mathrm{res}^G_H(\sO^{\vir}_{[G, W, E^\bullet]})  &= g^*(\sO^{\vir}_{[G, W, E^\bullet]}) \\
&=^1 g^*p_{E^\bullet}^! (\sO_{BG}) \\
&=^2 q_{E^\bullet}^!f^* (\sO_{BG}) \\ 
&= q_{E^\bullet}^!(\sO_{BH}) \\
&=^3 \sO^{\vir}_{[H, W, E^\bullet]},
\end{align*}
where $=^1$ and $=^3$ follow from \defref{defn:virtualstructuresheaf} and \cite[Definition 2.2]{QU}
and $=^2$ follows from the basic properties of 
virtual pullbacks \cite[Remark 2.3, Proposition 2.5]{QU}.
\end{proof}






We now prove a virtually smooth {\bl analogue} of the Edidin-Graham non-abelian localization theorem \cite[Theorem 5.1]{EG2}.
We show that under the Atiyah-Segal morphism the virtual structure sheaf on $X$ maps to the virtual structure sheaf of
the inertia scheme $I_X$ upto an automorphism induced by the virtual conormal bundle appearing in Lemma \ref{lem:vir-localization}. 
As we mentioned before the Atiyah-Segal map is independent of the choice of $g \in \psi$ (see \remref{remk:AS-map-indep-of-g}),
we prove in the following lemma that the image of the class of the virtual cornormal bundle in $I_{\psi}$ is also independent of 
the choice of $g \in \psi$.

\begin{lem}
Given a $G$-equivariant perfect obstruction theory $E^\bullet$ on $X$,
and $g, h \in \psi$,
$\theta_X^*(\Lambda_{-1}([N_{[i^h, E^\bullet]}^{\vir^\vee}])) = \Lambda_{-1}([N_{[i^g, E^\bullet]}^{\vir^\vee}])$
in $\KK_0(Z_g,X^g)$, where $\theta_X$ is defined in \eqref{diag:defnofphi_*0}. 
\end{lem}

\begin{proof}
Let $P_g$ and $P_h$ denote the cyclic subgroups of $G$ generated by $g$ and $h$, respectively.
The proof follows from the fact {\bl that the map $u: X^g \to X^h $ (see \remref{remk:AS-map-indep-of-g})} takes the fixed and moving parts under the $P_h$-action to the
fixed and moving parts under the $P_g$-action, respectively.
\end{proof}
In view of the above lemma and \eqref{eqn:Indep-0}, we denote $\mu^g_* (t_{g^{-1}} (\Lambda_{-1}([N_{[i^g, E^\bullet]}^{\vir^\vee}])))$,
the image of the class of the virtual cornormal bundle in $\KK_0(G, I^{\psi}_X)$,
by $\Lambda_{-1}([N_{[i^\psi, E^\bullet]}^{\vir^\vee}])$. Moreover, since $t_{g^{-1}}$ and $\mu^g_*$ are isomorphisms,
$\cap ~\Lambda_{-1}([N_{[i^\psi, E^\bullet]}^{\vir^\vee}]): \GG_0(G,I^\psi_X) \to \GG_0(G,I^\psi_X)$
is again an invertible map after localizing at $\fm_1$.
\begin{thm} \label{thm:Virtual-sh-AS}
Let $X \in \Sch^G_\C$ such that $G$ acts properly on $X$.  Suppose that there exists a $G$-equivariant closed immersion $X \inj M$ for some $M \in \Sm_\C^G$. Let $E^\bullet \to \LL_{X/k}$ be a $G$-equivariant perfect obstruction theory  which admits a global resolution and let 
$\tilde{E}^\bullet \to \LL_{I_X/k}$ be the induced $G$-equivariant perfect obstruction theory  on $I_X$ (see \lemref{lem:GPInertia}). 
Then under the Atiyah-Segal isomorphism
{\bl {$\vartheta^G_X: \GG_0(G,X) \to \GG_0(G,I_X)_{\fm_1}$}}, we have the following:
$$ \vartheta^G_X(\sO^\vir_{[G, X, E^\bullet]}) = 
\underset{\psi \in \Sigma^G_X} \sum \frac{\sO^\vir_{[G,I^\psi_X, \tilde{E}^\bullet]}}{\Lambda_{-1}([N_{[i^\psi, E^\bullet]}^{\vir^\vee}])}.
$$
\end{thm}

\begin{proof}
Fix an element $g \in \psi$ for each conjugacy class $\psi \in \Sigma^G_X$, under the notations of perfect obstruction theories given in 
\lemref{lem:GPInertia}, we have:
\begingroup
\allowdisplaybreaks
\begin{align*}
\vartheta^G_X(\sO^\vir_{[G, X, E^\bullet]})  &=^1 \underset{\psi \in \Sigma^G_X} \oplus \vartheta^{\psi}_X (\sO^\vir_{[G, X, E^\bullet]}) \\
																	 & =^2 \sum_{\psi} \mu^g_* (t_g^{-1} ({i^g_*}^{-1} (\mathrm{res}^G_{Z_g} (\sO^\vir_{[G, X, E^\bullet]})))) \\
																	 & =^3 \sum_{\psi} \mu^g_* (t_g^{-1} ({i^g_*}^{-1} (\sO^\vir_{[Z_g, X, E^\bullet]})))\\
																	 & =^4 \sum_{\psi} \mu^g_* \left( t_g^{-1} \left(
																	 \frac{\sO^\vir_{[Z_g, X^g,F_g^\bullet]}}{\Lambda_{-1}([N_{[i^g, E^\bullet]}^{\vir^\vee}])}\right) \right)\\
																	 & =^5 \sum_{\psi} \frac{\mu^g_*(\sO^\vir_{[Z_g, X^g,F_g^\bullet]})}
																	 {\mu^g_*(t_{g^{-1}}(\Lambda_{-1}([N_{[i^g, E^\bullet]}^{\vir^\vee}])))}\\
																	 & =^6 \sum_{\psi} \frac{\sO^\vir_{[G,I^\psi_X, \tilde{E}^\bullet]}} 
																	 {\mu^g_*(t_{g^{-1}}(\Lambda_{-1}([N_{[i^g, E^\bullet]}^{\vir^\vee}])))}\\
																	 & = \sum_{\psi} \frac{\sO^\vir_{[G,I^\psi_X, \tilde{E}^\bullet]}}{\Lambda_{-1}([N_{[i^\psi, E^\bullet]}^{\vir^\vee}])}
																	 ~,															
\end{align*}
\endgroup
where $=^1$, $=^2$ follow from Definitions \ref{defn:Atiyah-Segal-map}, \eqref{eqn:Atiyahsegal} and $=^3$, $=^4$  follow
from Lemmas \ref{lem:restriction-vir-sh}, \ref{lem:vir-localization}, respectively.
$=^5$ is a consequence of the fact that $F_g^{\bullet}$ is $<g>$-invariant and $=^6$ is given by \propref{prop:Moritainducedperfectobstructiontheory}.
\end{proof}

Recall from \cite[Definition 10.8]{KS20} that the Riemann-Roch isomorphism to the inertia scheme is defined as the composition of the Atiyah-Segal map and the
Riemann-Roch transformation of the inertia scheme:
$$I\tau_X := \tau^G_{I_X} \circ \vartheta^G_X : \GG_0(G,X) \to \bigoplus_{\psi \in \Sigma_G^X} \CH_*^G(I_X).$$
We have the following {\bl analogue} of \thmref{thm:main} relating the 
virtual structure sheaf of a $G$-scheme with a given $G$-equivariant perfect obstruction theory  
and the virtual {\bl fundamental }class of the inertia scheme corresponding to the naturally induced
$G$-equivariant perfect obstruction theory,
under the Riemann-Roch transformation to the inertia.

\begin{cor} \label{cor:Inertia-RR}
	Under the notations of Theorem \ref{thm:Virtual-sh-AS} further  suppose  assumptions in \ref{assmp} are satisfied. Then we have the following equality in $\CH_*^G(I_X)$:
$$
	I\tau^G_X (\sO^\vir_{[G, X, E^\bullet]}) = \underset{\psi \in \Sigma^G_X} \sum ch^G(\Lambda_{-1}([N_{[i^\psi, E^\bullet]}^{\vir^\vee}])^{-1}
	                                                                  Td^G(T^\vir_{[I^\psi_X, \tilde{E}^\bullet]}) \cap  [G,I_X^{\psi^\vir}, \tilde{E}^\bullet].
$$
\end{cor}

\begin{proof}
The proof follows directly from Theorems \ref{thm:Virtual-sh-AS} and \ref{thm:main}, by noting that
for any $\epsilon \in \KK_0(G,I^\psi_X)$ and 
$\alpha \in \GG_0(G, I^\psi_X)$, 
$\tau^G_{I^\psi_X} (\epsilon \alpha) = ch^G(\epsilon) \tau^G_{I^\psi_X} (\alpha)$ (cf. \cite[Theorem 3.1(c)]{EG1}).
\end{proof}


\begin{thebibliography}{99}
	\bibitem[AOV08]{AOV} Dan Abramovich, Martin Olsson, and Angelo Vistoli, {\sl Tame stacks in positive characteristic\/}, 
	Annales de l'institut Fourier, {\bf 58}, no. 4, (2008), 1057--1091.
	

\bibitem[BFM75]{BFM} Paul Baum, William Fulton, and Robert MacPherson, {\sl Riemann-Roch for 
	singular varieties\/}, Publ. Math. IHES, {\bf 45}, (1975), 101--145. \

\bibitem[BF97]{BF}
{ Kai Behrend and Barbara Fantechi,} {\sl The intrinsic normal cone},  Inventiones mathematicae, {\bf 128}, no. 1, (1997),  45--88.

\bibitem[BZN18]{BN16}
 David Ben-Zvi and David Nadler, {\sl Betti geometric langlands}, Proceedings of Symposia in Pure Mathematics, {\bf 97} (2018), 3-41. 

\bibitem[Br17]{Brion}
Michel Brion, {\sl Algebraic group actions on normal varieties\/}, Transactions of the Moscow Mathematical Society, {\bf 78}, (2017),  85--107.



\bibitem[CFK09]{CK}
Ionut Ciocan-Fontanine and Mikhail Kapranov, {\sl Virtual fundamental classes via dg–manifolds\/}, Geometry \& Topology, {\bf 13}, no. 3, (2009), 1779--1804.

\bibitem[De73]{Deligne}
 Pierre Deligne, {\sl La formule de dualité globale\/}, Th\'{e}orie des topos et cohomologie \'{e}tale des sch\'{e}mas, Springer, Berlin, Heidelberg, (1973),
 481--587.

\bibitem[EG98]{EG}
Dan Edidin  and William Graham, {\sl Equivariant intersection theory (With an appendix by Angelo Vistoli: The Chow ring of M2),} Inventiones mathematicae, 
{\bf 131},  no. 3, (1998),  595--634.


\bibitem[EG00]{EG1}
{ Dan Edidin and William Graham}, {\sl Riemann-Roch for equivariant Chow groups},  Duke Mathematical Journal, {\bf 102}, no. 3, (2000),  567--594.\/

\bibitem[EG05]{EG2} Dan Edidin  and William Graham, {\em Nonabelian localization in 
	equivariant K-theory and Riemann-Roch for quotients\/}, Adv. Math., 
{\bf 198}, no. 2, (2005), 547--582. \

\bibitem[EG08]{EG4} Dan Edidin  and William Graham, {\em Algebraic cycles and 
	completions of equivariant $K$-theory\/}, Duke Math. J., {\bf 144}, no. 3, 
(2008), 489--524. \

\bibitem[Jo07]{joshua}
 Roy Joshua, {\sl Bredon-style homology, cohomology and Riemann–Roch for algebraic stacks}, Advances in Mathematics, {\bf 209}, no. 1, (2007), 1--68.

\bibitem[FG10]{FG}
{ Barbara Fantechi and Lothar G\"{o}ttsche}, {\sl Riemann-Roch theorems and elliptic genus for virtually smooth schemes,} Geometry \& Topology, {\bf 14}, no. 1, (2010), 83--115.

\bibitem[GP99]{GP99}
 Tom Graber  and Rahul Pandharipande, {\sl Localization of virtual classes\/}, Inventiones mathematicae, {\bf 135}, no. 2 (1999), 487--518.
 
\bibitem[Gu20]{Gu20}
J\'{e}r\'{e}my Gu\'{e}r\'{e}, {\sl Congruences on K-theoretic Gromov--Witten invariants}, arXiv preprint, arXiv:2009.08485, (2020).

\bibitem[Ha18]{HL}
Daniel Halpern-Leistner, {\sl $\theta$-stratifications, $\theta$-reductive stacks, and applications\/}, Algebraic Geometry: Salt Lake City 2015, Proc. Sympos. Pure Math., {\bf 97.1}, Amer. Math. Soc., Providence, RI, (2018), 349--379. 


\bibitem[Kh19]{khan}
Adeel A. Khan, {\sl Virtual fundamental classes of derived stacks I}, arXiv preprint, arXiv:1909.01332, (2019).


\bibitem[KP19]{Kiem-Park}
Young-Hoon Kiem and Hyeonjun Park, {\sl Virtual intersection theories\/}, arxiv preprint, arXiv:1908.03340, (2019). 


\bibitem[KKP03]{KKP}
Bumsig Kim, Andrew Kresch, and Tony Pantev,  {\sl Functoriality in intersection theory and a conjecture of Cox, Katz, and Lee}, Journal of Pure and Applied Algebra, {\bf 179}, no. 1-2, (2003), 127--136.

\bibitem[Ko95]{KO}
Maxim Kontsevich, {\sl Enumeration of rational curves via torus actions}, In The moduli space of curves, pp. 335-368. Birkh\"auser Boston, 1995.

\bibitem[Kr99]{Kresch} Andrew Kresch,  {\em Cycle groups for Artin stacks\/},
Invent. Math., {\bf 138}, no. 3, (1999), 495--536. \







\bibitem[KR18]{KR18}
 Amalendu Krishna and Charanya Ravi, {\sl Algebraic K-theory of quotient stacks},
Annals of K-Theory, {\bf 3}(2), (2018), 207--233.


\bibitem[KS20]{KS20}
Amalendu Krishna and Bhamidi Sreedhar,  {\sl Atiyah-Segal theorem for Deligne-Mumford stacks and applications \/}  Journal of Algebraic Geometry {\bf 29}, no. 3 (2020), 403--470.


	\bibitem[Le04]{YPLee} Y.-P. Lee, {\em Quantum K-Theory I: Foundations,}  Duke Math. J., {\bf 121}, (2004), no. {3}, 389--424.
	
	\bibitem[LT98]{LT}
		Jun Li and Gang Tian, { \sl Virtual moduli cycles and Gromov-Witten invariants of algebraic varieties\/}, Journal of the American Mathematical Society {\bf 11}, no. 1 (1998), 119--174.
	
	
	
	\bibitem[Ma12]{Manolache}
	Cristina Manolache, {\sl Virtual pull-backs}, Journal of Algebraic Geometry, {\bf 21}, no. 2, (2012), 201--245.
	

	
	\bibitem[GIT]{GIT}  David Mumford, John Fogarty, and Frances Kirwan, {\sl Geometric
		Invariant Theory\/}, 3rd enlarged edition, Springer-Verlag, (1994).
	
		\bibitem[FMR20]{NMR20}
	 Nadir Fasola, Sergej Monavari and Andrea T. Ricolfi, {\sl Higher rank K-theoretic Donaldson-Thomas theory of points}, arXiv preprint arXiv:2003.13565, (2020).
	
	\bibitem[Ri20]{Ri20}
	Andrea T. Ricolfi,  {\sl The equivariant Atiyah class \/},  arXiv preprint arXiv:2003.05440, (2020).
	

\bibitem[Si04]{Si}
Bernd Siebert,  {\sl Virtual fundamental classes, global normal cones and Fulton’s canonical classes\/},  Aspects Math., E36, Friedr. Vieweg,
Wiesbaden, (2004), 341--358.

\bibitem[Th18]{Th18}
Richard Thomas, {\sl A K-theoretic Fulton class}, arXiv preprint arXiv:1810.00079, (2018).

\bibitem[Th20]{Thomas}
Richard  Thomas,  {\sl Equivariant K-theory and refined Vafa–Witten invariants.\/}, Communications in Mathematical Physics, (2020), 1--50.
	
	\bibitem[Th87]{Thomason}
	Robert W. Thomason, {\sl Algebraic K-theory of group scheme actions\/}, Annals  Of Mathematics Studies, {\bf 113},  (1987), 539-563.
	
	\bibitem[To14]{Tonita}
	Valentin Tonita, {\sl A virtual Kawasaki–Riemann–Roch formula\/},  Pacific Journal of Mathematics, {\bf 268}, no. 1, (2014),  249--255.
	
	\bibitem[To99]{Totaro}
	Burt Totaro, {\sl The Chow ring of a classifying space\/}, Proceedings of symposia in pure mathematics, {\bf 67}, Amer. Math. Soc., Providence, 
	RI, (1999), 249-284
	
	\bibitem[Qu18]{QU} Feng Qu, {\sl  Virtual  Pullbacks In $K$-Theory\/}, Ann. Inst. Fourier, Grenoble, {\bf 68}, no. 4, (2018), 1609--1641.\
	

	
	\bibitem[Vi89]{Vistoli}Angelo Vistoli, {\sl Intersection theory on algebraic stacks and on their moduli spaces\/}, Inventiones mathematicae, {\bf 97}, no. 3, (1989), 613--670.
	
		\bibitem[Vi91]{V}
	Angelo Vistoli, {\sl Higher equivariant $ K $-theory for finite group actions\/}, Duke Mathematical Journal, {\bf 63}, no. 2 (1991),  399--419.
	
\end{thebibliography}
\end{document}